\documentclass[11pt]{amsart}
\usepackage[margin=1in]{geometry}
\usepackage{graphicx}
\usepackage{bm}
\usepackage{multirow}
\usepackage{mathtools}
\usepackage{amsmath,amssymb,amsfonts}
\usepackage{xcolor}

\definecolor{exact}{RGB}{166, 97, 26}
\definecolor{no}{RGB}{223, 194, 125}
\definecolor{old}{RGB}{128, 205, 193}
\definecolor{new}{RGB}{1, 133, 113}

\newtheorem{theorem}{Theorem}[section]

\newtheorem{proposition}{Proposition}[section]

\theoremstyle{definition}

\theoremstyle{definition}
\newtheorem{remark}[theorem]{Remark}

\newcommand{\RN}[1]{%
	\textup{\uppercase\expandafter{\romannumeral#1}}%
}
\newcommand{\bb}[1]{\mathbf{#1}}

\newcommand{\lr}[1]{\left(#1\right)}
\newcommand{\rb}[2]{\left\{#1,#2\right\}}
\newcommand{\ib}[2]{\left[#1,#2\right]}
\newcommand{\ee}{\bb{e}}
\newcommand{\uu}{\bb{U}}
\newcommand{\ut}{\bb{U}^\intercal}
\newcommand{\yy}{\bb{y}}
\newcommand{\xx}{\bb{x}}
\newcommand{\ww}{\bb{w}}
\newcommand{\Am}{\bb{A}}
\newcommand{\LL}{\bb{L}}
\newcommand{\Lh}{\hat{\bb{L}}}
\newcommand{\Lb}{\bar{\bb{L}}}

\newcommand{\MM}{\bb{M}}
\newcommand{\Mh}{\hat{\bb{M}}}
\newcommand{\Mb}{\bar{\bb{M}}}
\newcommand{\vv}{\bb{v}}
\newcommand{\am}{\bb{a}_{k_0}}
\newcommand{\ah}{\hat{\bb{a}}_{k_0}}
\newcommand{\qq}{\bb{q}}
\newcommand{\pp}{\bb{p}}
\newcommand{\sk}{\bb{s}_{k_1}}
\newcommand{\hsk}{\hat{\bb{s}}_{k_1}}

\newcommand{\xt}{\tilde{\bb{x}}}
\newcommand{\xh}{\hat{\bb{x}}}
\newcommand{\nsh}{\nabla\hat{S}}
\newcommand{\neh}{\nabla\hat{E}}
\newcommand{\xih}{\hat{\bm{\xi}}}
\newcommand{\zh}{\hat{\bm{\zeta}}}
\newcommand{\norm}[1]{\left\lVert#1\right\rVert}
\newcommand{\nn}[1]{\left|#1\right|}

\title{Energetically Consistent Model Reduction for Metriplectic Systems}
\author{Anthony Gruber$^{1,*}$}
\author{Max Gunzburger$^1$}
\author{Lili Ju$^2$}
\author{Zhu Wang$^2$}
\thanks{$^*$Corresponding author: (Anthony Gruber)  anthony.gruber.d@gmail.com}

\email{anthony.gruber.d@gmail.com, mgunzburger@fsu.edu, ju@math.sc.edu, wangzhu@math.sc.edu}

\address{$^1$ Department of Scientific Computing, Florida State University, 400 Dirac Science Library, Tallahassee, FL 32306, USA}

\address{$^2$ Department of Mathematics, University of South Carolina, 1523 Greene Street, Columbia, SC 29208, USA}

\begin{document}

\begin{abstract}
The metriplectic formalism is useful for describing complete dynamical systems which conserve energy and produce entropy.  This creates challenges for model reduction, as the elimination of high-frequency information will generally not preserve the metriplectic structure which governs long-term stability of the system.  Based on proper orthogonal decomposition, a provably convergent metriplectic  reduced-order model is formulated which is guaranteed to maintain the algebraic structure necessary for energy conservation and entropy formation.  Numerical results on benchmark problems show that the proposed method is remarkably stable, leading to improved accuracy over long time scales at a moderate increase in cost over naive methods.
\vspace{0.5pc}


\end{abstract}


\maketitle

\section{Introduction}




Metriplectic dynamical systems offer a prototypical example of how the algebraic structure internal to a system can govern the behavior of its observable quantities.  Also referred to as GENERIC systems (see \cite{ottinger2006}), metriplectic dynamics are produced through a combination of reversible and irreversible contributions whose constituent parts are a noncanonical Poisson structure and a degenerate Riemannian metric structure, respectively.  Mathematically, these structures are reflected by algebraic brackets which formally separate the dynamics into terms that are ``energy-preserving'' and terms that are ``dissipative'' (see Section \ref{sec:overview}).  Combined with appropriate compatibility (or degeneracy) conditions, this seemingly simple idea is geometrically rich and encodes a strong form of the first and second laws of thermodynamics, making it powerful enough to represent many physical systems of interest (see Section \ref{sec:rw}).  On the other hand, standard computationally efficient reduced-order models (ROMs) for these systems based on purely statistical considerations will not generally preserve the rich structure afforded by the metriplectic formalism, which can lead to unreasonable or unrealistic results in real-time use cases (see e.g. Section~\ref{sec:numerics}).  To elucidate the benefits of metriplectic structure-preservation in the context of model reduction, a genuinely metriplectic ROM based on proper orthogonal decomposition (POD) is proposed in Theorem~\ref{thm:ROM} which is shown in Theorem~\ref{thm:conv} to converge to the true solution as the reduced dimension increases.  The remainder of the manuscript is dedicated to a detailed description of this ROM along with an evaluation of its performance on benchmark examples.




\subsection{Overview}\label{sec:overview}
It is first useful to recall metriplectic systems in more detail.  The generator for metriplectic dynamics is a notion of free energy $F=E+S$ described by functions $E,S: \mathcal{P}\to\mathbb{R}$ (representing energy and entropy, respectively) which are defined on some phase space $\mathcal{P}$ that may be finite or infinite dimensional. In this case, any observable quantity $\bb{O}:\mathcal{P}\to \mathbb{R}^N$ (for some $N$) evolves as
\[ \dot{\bb{O}} = \rb{\bb{O}}{F} + \ib{\bb{O}}{F} = \rb{\bb{O}}{E} + \ib{\bb{O}}{S}, \]
with respect to a time-dependent state $\xx:[0,T]\to \mathcal{P}$, where $\rb{\cdot}{\cdot}$ is a noncanonical Poisson bracket on $\mathcal{P}$ capturing the reversible dynamics and $\ib{\cdot}{\cdot}$ is a degenerate metric bracket on $\mathcal{P}$ capturing the irreversible dynamics.  Metriplectic structure is enforced by the implicit degeneracy conditions $\rb{S}{\cdot} = \ib{E}{\cdot} = \bm{0}$, which guarantee an analogue of energy conservation and entropy production.  To describe this more precisely, recall that the Poisson structure $\rb{\cdot}{\cdot}$ is a Lie algebra realization on functions and so is bilinear and skew-symmetric (SS), while the degenerate metric structure $\ib{\cdot}{\cdot}$ is chosen to be bilinear and symmetric positive semi-definite (SPSD).  This allows for concrete expression of the reversible and irreversible brackets as
\begin{align*}
    \rb{\bb{O}}{E} &= \nabla\bb{O}\cdot\LL\nabla E, \\
    \ib{\bb{O}}{S} &= \nabla\bb{O}\cdot\MM\nabla S,
\end{align*}
where $\cdot$ represents a choice of inner product on $\mathcal{P}$, $\nabla$ denotes the gradient with respect to $\cdot$ defined through $dF(\bb{v}) = \nabla F \cdot \bb{v}$, and $\LL,\MM:\mathcal{P}\to\mathcal{P}$ are SS resp. SPSD linear operators which may depend on the state $\xx$.  Here again no distinction is made between finite and infinite dimensional systems, as this affects only the choice of inner product $\cdot$.  In many cases of interest the observable $\bb{O}\lr{\xx} = \xx$ is simply the identity, so that the system above further simplifies to the standard equations for metriplectic dynamics \cite{morrison1984,ottinger2006},
\begin{equation}\label{eq:FOM}
    \dot{\xx} = \rb{\xx}{E} + \ib{\xx}{S} = \LL\nabla E + \MM\nabla S,
\end{equation}
which in view of the compatibility conditions
\begin{equation}\label{eq:compatibility}
\LL\nabla S = \MM\nabla E = \bm{0},
\end{equation}
preserve a strong form of the first and second thermodynamical laws.  In particular, since $\LL^\intercal=-\LL$, 
\[ \dot{E} = \dot{\xx}\cdot\nabla E = \LL\nabla E\cdot \nabla E + \MM\nabla S \cdot \nabla E = \nabla S \cdot \MM\nabla E = 0, \]
so that the energy $E$ is conserved along the evolution.  Similarly, the fact that $\MM^\intercal=\MM$ is SPSD implies the relationship
\[ \dot{S} = \dot{\xx} \cdot \nabla S = \LL\nabla E \cdot \nabla S + \MM\nabla S \cdot \nabla S = -\nabla E \cdot \LL\nabla S + \MM\nabla S \cdot \nabla S = |\nabla S|^2_M \geq 0,\]
so that the entropy $S$ is nondecreasing.  Here it becomes clear that asymptotic stability is built-in to the metriplectic framework, as choosing $-S$ as a Lyapunov function shows that 
solutions to \eqref{eq:FOM} will naturally relax to the state $\nabla F = \bm{0}$.  Moreover, this gives a degree of freedom in describing a physical system with metriplectic structure, as $S$ can be chosen judiciously from the Casimirs of the Poisson bracket i.e. those functions which annihilate it.  Geometrically, it is interesting to observe that the motion of $\xx$ is everywhere tangent to the level curves of $E$ and transverse to those of $S$, which is reflective of the fact that metriplectic dynamics are a combination of Hamiltonian and generalized gradient flows.  When $\MM=\bm{0}$, $E=H$ is the Hamiltonian function, and $\LL=\bb{J}$ is a square root of $-\bb{I}$ (note the freedom in sign),  \eqref{eq:FOM} reduces to Hamilton's equations of motion $\dot{\xx} = \rb{\xx}{H} = \bb{J}\nabla H$.  Similarly, when $\LL=\bm{0}$ and $S = -G$ for some $G:\mathcal{P}\to\mathbb{R}$, \eqref{eq:FOM} reduces to a generalized gradient flow $\dot{\xx} = -\ib{\xx}{G} = -\bb{M}\nabla G$.





\subsection{Related Work}\label{sec:rw}
The metriplectic/GENERIC forms of many physical systems have already been proposed and studied theoretically for some time.  The compressible Navier-Stokes equations were seen to be metriplectic in \cite{morrison1984}, and general complex fluids were incorporated into the formalism in \cite{grmela1997}.  This paved the way for the inclusion of other physical phenomena such as those based on Korteweg-type fluids \cite{suzuki2020} and the Smoluchowski equation for colloidal suspensions \cite{wagner2001}.  Moreover, a constrained GENERIC rheological model for polymer solutions was developed in \cite{ait1999} and shown to be effective in predicting steady shear viscosity, while a formulation of dissipative magnetohydrodynamics was discovered in \cite{materassi2012} and used in studying two-dimensional incompressible plasma flow.  Beyond fluids, metriplectic structure has also been useful in describing mechanical systems such as three-dimensional rigid body dynamics \cite{materassi2018}, Hamiltonian systems with friction \cite{caligan2016}, a Vlasov-Fokker-Planck equation \cite{duong2013}, and others based on large deviation principles in physics.  

There have been far fewer works addressing the computational aspects of metriplectic systems, though some noteworthy progress has been made.  Structure preserving numerical methods for finite strain thermoelastodynamics in GENERIC form are discussed in \cite{betsch2019}, where so-called Energy-Momentum-Entropy consistent schemes are shown to increase stability of the discrete system.  A compatible discretization for GENERIC problems using finite elements in space and a monolithic integrator in time was developed in \cite{romero2010} and applied to nonlinear problems in thermoelasticity, again demonstating improved stability properties.  There is also a promising line of research into metriplectic integrators using neural network technology, which has produced works such as \cite{lee2021,zhang2021}.

From the perspective of model reduction, it has long been recognized that computational models perform better when informed by the algebraic structure of the systems that they are modeling.  This remains true for low-dimensional approximation, where the model being approximated is itself a surrogate for some physical phenomena.  This has produced an entire subfield of structure-preserving model reduction, whose goal is to design effective low-fidelity surrogates which preserves desired properties of the high-fidelity model under consideration.  The relative ubiquity and rich mathematical structure of Hamiltonian systems has inspired several works on Hamiltonian structure-preserving ROM  \cite{peng2016,gong2017,afkham2017,maboudi2019,sockwell2019}, as well as numerous extensions to the port-Hamiltonian formalism \cite{polyuga2010,beattie2011,gugercin2012,chaturantabut2016,liljegren2020} which is a useful generalization of its namesake where the Hamiltonian structure on the interior is allowed to interface with general ``ports'' on the boundary.

\begin{remark}
In fact, the system \eqref{eq:FOM} can be embedded into the port-Hamiltonian formalism,
\begin{equation*}
    \begin{split}
        \dot{\xx} &= (\bb{J}-\bb{R})\nabla H(\xx) + \bb{B}\bb{u}(t),\\
        \yy &= \bb{B}^\intercal\nabla H(\xx),
    \end{split}
\end{equation*}
where $\bb{J}$ is SS and $\bb{R}$ is SPSD. In particular, decompose $\MM = \bb{CDC}^\intercal$.  Then, choosing $\bb{R}=\bm{0}$, $\bb{J}=\bb{L}$ $\bb{B}=\bb{C}, \bb{u}=-\bb{D}\yy$, and $H=E-S$ the exegy function of the system, it follows that 
\begin{align*}
    \dot{\xx} = \LL\left(\nabla E - \nabla S\right) - \bb{CDC}^\intercal\left(\nabla E - \nabla S\right) = \LL\nabla E + \MM\nabla S,
\end{align*}
since $\LL\nabla S = \MM\nabla E = \bm{0}$.  On the other hand, none of the port-Hamiltonian ROM work known to the authors can guarantee preservation of the degeneracy conditions \eqref{eq:compatibility} necessary for metriplectic structure.
\end{remark}

Apart from Hamiltonian systems and their extensions, significant work involving structure-preserving ROM has also been done on topics such as moment-preserving Krylov subspace projection \cite{bai2005} and Lagrangian variational problems \cite{lall2003}.  An interpolatory model reduction strategy preserving symmetry, higher order structure, and state constraints is discussed in \cite{beattie2009}, and a ROM for damped wave propagation in transport networks is developed in \cite{egger2018}.  It is remarkable that the strategy in \cite{egger2018} is similar to ours in that the preservation of algebraic compatibility conditions at the reduced level assures desired properties such as conservation of mass, dissipation of energy, passivity, and existence of steady states at the full resolution.

\section{Preliminaries}\label{sec:pre}

To describe the present method for metriplectic model reduction, it is useful to review some basics of POD-based ROMs.  First, recall that the goal is to study systems which conserve some notion of energy $E$, so it is beneficial to express any approximation $\xt\approx\xx \in \mathbb{R}^N$ to the full-order state as a perturbation from some reference configuration $\xx_0 \in \mathbb{R}^N$, i.e. $\xt = \xx_0 + \uu\xh$ where $\xh\in\mathbb{R}^n$ and $\uu:\mathbb{R}^n\to\mathbb{R}^N$.  This ensures the true value of $E$ is exactly preserved at least at the point where $\xh = \bm{0}$, which serves as the initial condition for the reduced-order system.

Consider the standard POD-ROM procedure with this in mind.  Let $\xx\in\mathbb{R}^N$ be a semi-discrete object representing the solution to a system of $N\in\mathbb{N}$ ODEs, and let $\bb{Y}\in\mathbb{R}^{N\times n_t}$ be a matrix with rank $r\leq\min\{N,n_t\}$ containing snapshots of the high-fidelity solution $\bb{w} = \xx-\xx_0$ at $n_t$ discrete points in the interval $[0,T]$ where $T\in\mathbb{R}$ represents the final simulation time.  If $\bb{Y} = \tilde{\uu}\bm{\Sigma}\bb{V}^\intercal$ is the singular value decomposition, standard computations show that the matrix $\uu\in \mathbb{R}^{N\times n}$ comprised of the first $n<r$ columns of $\tilde{\uu}$ minimizes the $L^2\lr{[0,T]}$ reconstruction error of $\bb{w}$, and that this error is precisely the sum of the remaining squared singular values \cite{liang2002}.  More precisely, it follows that 
\[ \norm{\bb{w} - \uu\ut\bb{w}}^2 := \int_0^T \nn{\bb{w} - \uu\ut\bb{w}}^2\,dt = \sum_{i=n+1}^r \sigma_i^2, \]
where $\sigma_i$ is the $i^{th}$ singular value of $\bb{Y}$.  This is the basis for the standard POD-ROM procedure, which is applied to the system governing $\xx$ by making the approximation $\xt = \xx_0 + \uu\xh$ and using that $\ut\uu = \bb{I}$ in $\mathbb{R}^n$.  In the case of the metriplectic system \eqref{eq:FOM}, this yields the reduced-order model
\begin{equation}\label{eq:nospROM}
    \dot{\xh} = \ut\LL(\xt)\nabla E(\xt) + \ut\MM(\xt)\nabla S(\xt),
\end{equation}
which is the system of $n$ scalar ODEs that best approximates the FOM \eqref{eq:FOM} in the above sense, but clearly does not preserve the compatibility conditions \eqref{eq:compatibility} necessary for metriplectic structure.  As will be seen in the numerical experiments (see Section~\ref{sec:numerics}), this creates instability which can lead to unphysical blow-up of the solution in time. 


\begin{remark}
For notational convenience, dependence on the states $\xx,\xt,\xh$ is suppressed when the context is clear.  Similarly, the Einstein summation convention is adopted so that any tensor index appearing both up and down in an expression is summed over its appropriate range.
\end{remark}

As a first attempt at remedying this, it is reasonable to consider searching for mappings $\Lh,\Mh$ depending only on $\xh$ such that 
\begin{equation}\label{eq:ideal}
    \ut\LL = \Lh\ut, \qquad \ut\MM = \Mh\ut.
\end{equation}
This would convert \eqref{eq:nospROM} into the best possible ROM,
\[ \dot{\xh} = \ut\LL(\xt)\nabla E(\xt) + \ut\MM(\xt)\nabla S(\xt) = \Lh(\xh)\neh\lr{\xh} + \Mh(\xh)\nsh\lr{\xh}, \]
where we have introduced the following notation for real-valued functions $F:\mathbb{R}^N\to\mathbb{R}$,
\[ \hat{F} = F\circ\xt, \qquad \nabla\hat{F} = \xt'\cdot\nabla F = \ut\nabla F. \]
Note that the compatibility conditions \eqref{eq:compatibility} are automatically satisfied in this case, as $\Lh\nsh = \Lh\ut\nabla S = \ut\LL\nabla S = \bb{0}$ and similarly for $\Mh\neh$.  On the other hand, \eqref{eq:ideal} is an overdetermined system of equations when $N>n$, and solving the normal equations gives only the system
\begin{equation}\label{eq:oldrom}
    \dot{\xh} = \Lb\neh + \Mb\nsh, \qquad \Lb = \ut\LL\uu, \qquad \Mb = \ut\MM\uu.
\end{equation}
which has the advantage of informing the low-dimensional system with the symmetry relationships $\Lb^\intercal = -\Lb$ and $\Mb^\intercal = \Mb$, but still cannot guarantee metriplectic structure preservation.  Since $\uu\ut \neq \bb{I}$, this gives only 
\begin{align*}
    \Lb\nsh &= \ut\LL\uu\ut\nabla S \neq \bm{0}, \\
    \Mb\neh &= \ut\MM\uu\ut\nabla E \neq \bm{0},
\end{align*}
so that the first and second laws of thermodynamics remain violated.  Note that in the case $\LL=\bm{0}$ or $\MM=\bm{0}$, the conditions \eqref{eq:compatibility} are vacuous and \eqref{eq:oldrom} provides a useful reduced order model which preserves some structure present in the original system.  In fact, this ROM in the case $\MM=\bm{0}$ is precisely the Hamitonian structure-preserving ROM proposed in \cite{gong2017}.



\subsection{Metriplectic Structure Preservation}
One way to preserve metriplectic structure is to incorporate the compatibility conditions \eqref{eq:compatibility} explicitly.  Let $\bb{e}_k$ denote the $k^{th}$ standard basis vector in $\mathbb{R}^N$ and denote $\nabla F = F^k\bb{e}_k$ where $F^k = \bb{e}_k\cdot\nabla F = \partial F/\partial x^k$.  For each $1\leq i,j \leq N$, consider solving the underdetermined equations
\begin{equation}\label{eq:tens}
\begin{split}
    L_{ij} &= \xi_{ijk}S^k, \\
    M_{ij} &= \zeta_{ikjl}E^kE^l,
\end{split}
\end{equation}
for tensors $\bm{\xi},\bm{\zeta}$ in $\lr{\mathbb{R}^N}^{\otimes 3},\lr{\mathbb{R}^N}^{\otimes 4}$ respectively which may depend on the state $\xx$ and which satisfy the symmetry relations
\begin{equation}\label{eq:syms}
\begin{split}
    \xi_{ijk} &= -\xi_{jik} = -\xi_{ikj}, \\
    \zeta_{ikjl} &= -\zeta_{kijl} = -\zeta_{iklj} = \zeta_{jlik}.
\end{split}
\end{equation}
This is always possible as long as $\nabla E,\nabla S$ are nonzero for all $\xx$ (c.f. Proposition~\ref{prop:tensorcomp}), and otherwise the metriplectic structure is degenerate.  Notice that \eqref{eq:syms} is simply the coordinate-wise expression of total antisymmetry in $\bm{\xi}$ as well as symmetric $12-34$ pairwise antisymmetry in $\bm{\zeta}$.  The advantage of this approach is that the degeneracy conditions now follow immediately from the symmetries \eqref{eq:syms}.  Identifying $\mathbb{R}^N$ with its dual to make use of the canonical ``index-raising'' isomorphism, for any $1\leq i\leq N$ it follows that
\begin{align*}
    \left(\LL\nabla S\right)^i &= \xi^i_{jk}S^kS^j = \xi^i_{kj}S^jS^k = -\xi^i_{jk}S^jS^k = 0,\\
    \left(\MM\nabla E\right)^i &= \zeta^i_{kjl}E^kE^lE^j = -\zeta^i_{kjl}E^kE^lE^j = 0,
\end{align*}
since in either case there is a contraction of the same vector over an antisymmetric pair of indices.  Therefore, if reduced-order objects $\xih\in\lr{\mathbb{R}^n}^{\otimes 3},\zh\in\lr{\mathbb{R}^n}^{\otimes 4}$ which preserve \eqref{eq:syms} can be found, the degeneracy conditions \eqref{eq:compatibility} will hold by construction.  


To that end, notice that \eqref{eq:FOM} and \eqref{eq:nospROM} can be rewritten respectively as 
\begin{align*}
    \dot{\xx} &= \bm{\xi}\left(\nabla S\right)\nabla E + \bm{\zeta}\left(\nabla E,\nabla E\right)\nabla S, \\
    \dot{\xh} &= \ut\bm{\xi}\left(\nabla S\right)\nabla E + \ut\bm{\zeta}\left(\nabla E,\nabla E\right)\nabla S,
\end{align*}
where the tensors $\bm{\xi},\bm{\zeta}$ are written suggestively to indicate their role as matrix-valued mappings.  This encourages the search for reduced-order matrices 
\begin{align*}
   \Lh &= \xih\lr{\nsh}, \\
    \Mh &= \zh\lr{\neh,\neh},
\end{align*} 
defined in terms of reduced tensors $\xih,\zh$ which satisfy 
\begin{equation}\label{eq:normal}
\begin{split}
    \ut\bm{\xi} &= \xih\lr{\ut}\ut, \\
    \ut\bm{\zeta} &= \zh\lr{\ut,\ut}\ut,
\end{split}
\end{equation} 
or, in coordinates and with $1\leq a,b,c,d \leq n$,
\begin{align*}
    U^a_i\xi^i_{jk} &= \hat{\xi}^a_{bc}U^b_jU^c_k, \\
    U^a_i\zeta^{i}_{kjl} &= \hat{\zeta}^a_{cbd}U^c_kU^d_lU^b_j.
\end{align*}
This is generally impossible when $N>n$ for the same reason as before, but \eqref{eq:normal} can again be interpreted as normal equations whose solutions yield
\begin{align*}
    \xih &= \ut\bm{\xi}\left(\uu\right)\uu, \\
    \zh  &= \ut\bm{\zeta}\left(\uu,\uu\right)\uu,
\end{align*}
or, once more in coordinate language,
\begin{align*}
    \hat{\xi}^a_{bc} &= U^a_i\xi^i_{jk}U^j_bU^k_c, \\
    \hat{\zeta}^a_{cbd} &= U^a_i\zeta^i_{kjl}U^k_cU^l_dU^j_b.
\end{align*}
It is straightforward to check that $\xih,\zh$ computed this way satisfy the necessary symmetries \eqref{eq:syms}, in which case the ROM
\begin{equation}\label{eq:ROM}
    \dot{\xh} = \rb{\xh}{\neh\lr{\xh}} + \ib{\xh}{\nsh\lr{\xh}} := \Lh(\xt)\neh\lr{\xh} + \Mh(\xt)\nsh\lr{\xh},
\end{equation}
will preserve the original metriplectic structure by construction.  Suppressing dependence on the state, it follows from symmetry considerations as before that 
\begin{align*}
    \hat{\LL}\nsh &= \xih\lr{\nsh}\nsh = \bm{0}, \\
    \Mh\neh &= \zh\lr{\neh,\neh}\neh = \bm{0},
\end{align*}
so that the first and second laws of thermodynamics become
\begin{align*}
    \dot{\hat{E}} &= \dot{\xh}\cdot\neh = \Lh\neh\cdot\neh + \nsh\cdot\Mh\neh = 0, \\
    \dot{\hat{S}} &= \dot{\xh}\cdot\nsh = -\neh\cdot\Lh\nsh + \Mh\nsh\cdot\nsh = \left|\nsh\right|^2_{\hat{M}} \geq 0,
\end{align*}
as desired.

\begin{remark}\label{rem:notmet}
It should be mentioned that the ``metriplectic-preserving'' moniker used to describe \eqref{eq:ROM} is a slight abuse of terminology in the case $\LL=\LL(\xx)$, as the reduced Poisson bracket $\{\cdot,\cdot\}$ generated by $\xih = \xih(\xx)$ is not guaranteed to satisfy the Jacobi identity for arbitrary $F,G,H:\mathbb{R}^n\to\mathbb{R}$,
\[ \rb{F}{\rb{G}{H}} + \rb{G}{\rb{H}{F}} + \rb{H}{\rb{F}{G}} = \bm{0}.\]
In fact, a direct calculation shows that the above Jacobi identity is equivalent to the statement (sum on $l$)
\[ \hat{L}_{il}\hat{L}_{jk,l}+\hat{L}_{jl}\hat{L}_{ki,l}+\hat{L}_{kl}\hat{L}_{ij,l} = 0, \qquad 1\leq i,j,k\leq n,\]
which may not vanish if $\LL(\xx)$ is state-dependent.  Conversely, the same calculation shows that the ROM \eqref{eq:ROM} is truly metriplectic outside of this case, since the terms involving second derivatives cancel solely due to symmetry properties. 
\end{remark}

\section{Computing the Metriplectic ROM}

To make use of the metriplectic ROM \eqref{eq:ROM}, it is necessary to have a reasonable way to compute $\bm{\xi},\bm{\zeta}$ and their reduced-order counterparts $\xih,\zh$.  Note that it is prohibitively expensive to compute general $3^{rd}$ and $4^{th}$-order tensors online even for moderately large $N$, as the number of tensor entries is exponential in the degree.  Therefore, it is necessary to find an efficient way to express \eqref{eq:ROM} which does not require the explicit construction of these tensors.  To that end, recall that $\MM$ is SPSD and so has the eigenvalue decomposition
\[ \MM = \sum_{\alpha=1}^r \lambda_\alpha\, \bb{m}^\alpha\otimes\bb{m}^\alpha, \]
where $1\leq r\leq N$ and all $\lambda_\alpha>0$ are positive.  If it is possible to write $\bb{m}^\alpha = \Am^\alpha\nabla E$ for some SS matrices $\Am^\alpha$, the tensor
\[ \bm{\zeta} = \sum_{\alpha=1}^r \lambda_\alpha\,\Am^\alpha\otimes \Am^\alpha, \]
would satisfy the desired conditions \eqref{eq:tens} and \eqref{eq:syms}.  The next result shows that this can always be done for metriplectic systems.


\begin{proposition}\label{prop:tensorcomp}
Let $k_0, k_1$ be indices such that $E^{k_0}\neq 0$ and $S^{k_1}\neq 0$.  Suppose $\MM = \sum_{\alpha}\lambda_\alpha\,\bb{m}^\alpha\otimes\bb{m}^\alpha$ and $\bb{B},\bb{C}$ are tensors such that $B^\alpha_{ik_0} = m^\alpha_i/E^{k_0}$, $C_{ijk_1} = L_{ij}/S^{k_1}$, and $B^\alpha_{ik}=C_{ijk}=0$ otherwise.  Then, the tensors $\bm{\xi}$ and $\bm{\zeta} = \sum_\alpha \lambda_\alpha\,\Am^\alpha\otimes \Am^\alpha$ with components
\begin{align*}
    \xi_{ijk} &= \frac{1}{2}\left(C_{ijk}+C_{jki}+C_{kij}-C_{jik}-C_{kji}-C_{ikj}\right), \\
    A^\alpha_{ik} &= B^\alpha_{ik}-B^\alpha_{ki},
\end{align*}
satisfy \eqref{eq:tens} and \eqref{eq:syms}.
\end{proposition}
\begin{proof}
This is a direct consequence of the compatibility conditions $\LL\nabla S = \MM\nabla E = \bm{0}$.  More precisely,
Let $\bb{B},\bb{C}$ be as in the statement of the Proposition.  First, recall that all eigenvectors $\bb{m}^\alpha$ of $\MM$ are linearly independent with positive eigenvalues $\lambda_\alpha>0$.  The compatibility condition $\MM\nabla E =\bm{0}$ then becomes
\[ \MM\nabla E = \sum_{\alpha=1}^r \lambda_\alpha\left(\bb{m}^\alpha\cdot\nabla E\right)\bb{m}^\alpha = \bm{0}, \]
so it follows that $\bb{m}^\alpha\cdot\nabla E=0$ for all $1\leq\alpha\leq r$.
Using $[F]$ to denote the indicator function of the statement $F$, the definition of $\Am$ then implies that for all $\alpha,i$,
\begin{align*}
    A^\alpha_{ik}E^k = \left(B^\alpha_{ik}-B^\alpha_{ki}\right)E^k = m^\alpha_i - [i=k_0]\,B^\alpha_{ki}E^k = m^\alpha_i - [i=k_0]\,\frac{\bb{m}^\alpha \cdot \nabla E}{E^{k_0}} = m_i^\alpha,
\end{align*}
which establishes that $\Am^\alpha\nabla E = \bb{m}^\alpha$ for all $\alpha$.  It follows that for all $1\leq i,j\leq N$,
\[ \zeta_{ikjl}E^kE^l = \sum_{\alpha=1}^r\lambda_\alpha A^\alpha_{ik} A^\alpha_{jl}E^kE^l = \sum_{\alpha=1}^r\lambda_\alpha m^\alpha_i m^\alpha_j = M_{ij},\]
which establishes the second part of \eqref{eq:tens}.  The corresponding symmetry relationship in \eqref{eq:syms} follows immediately from the skew-symmetry of $\Am^\alpha$ and the definition of $\bm{\zeta}$ as a sum of symmetric products.  Moving to the case of $\bm{\xi}$, it is straightforward to compute
\begin{align*}
    2\xi_{ijk}E^k &= \left(C_{ijk}+C_{jki}+C_{kij}-C_{jik}-C_{kji}-C_{ikj}\right)S^k\\ 
    &= \left(L_{ij}-L_{ji}\right) + [i=k_1]\left(C_{jki}-C_{kji}\right)S^k + [j=k_1]\left(C_{kij}-C_{ikj}\right)S^k \\
    &= 2\left(L_{ij} + [i=k_1]\,\frac{\left(\LL\nabla S\right)_j}{S^{k_1}} - [j=k_1]\,\frac{\left(\LL\nabla S\right)_i}{S^{k_1}}\right) = 2L_{ij},
\end{align*}
which follows since $\LL\nabla S = \bm{0}$. This establishes \eqref{eq:tens}, and the antisymmetry condition \eqref{eq:syms} is immediate since $\bm{\xi}$ is a multiple of the antisymmetrization of $\bb{C}$.
\end{proof}

\begin{remark}
It is interesting to note that a weaker form of Proposition~\ref{prop:tensorcomp} remains true in the general case of any 4-tensor $\bm{\zeta}$ satisfying the symmetries \eqref{eq:syms}.  In particular, it follows that $\bm{\zeta}$ decomposes as the sum of at most $N^2$ outer products with antisymmetric factors.  This is a consequence of a simple reshaping argument:  Consider the column-wise unfolding of $\bm{\zeta}$, denoted $\bar{\bm{\zeta}}$ and defined componentwise as
\[ \bar{\zeta}_{(k-1)N+i,(l-1)N+j} = \zeta_{ikjl}, \]
which is an $N^2\times N^2$ matrix that is symmetric by the 12-34 interchange symmetry of $\bm{\zeta}$.  Its spectral decomposition is 
\[\bar{\bm{\zeta}} =\sum_{\alpha=1}^{s} \mu_\alpha\, \bb{w}^\alpha \otimes \bb{w}^\alpha,\]
where $1\leq s\leq N^2$ and $\bb{w}^\alpha\cdot \bb{w}^\beta = \delta^{\alpha\beta}$.  So, if $\Am^{\alpha}$ is the folding of $\bb{w}^\alpha$ for all $1\leq\alpha\leq s$ i.e. $A^\alpha_{ij} = w^\alpha_{(j-1)N + i}$ for all $1\leq i,j\leq N$, it follows that $\Am^\alpha:\Am^\beta = \delta^{\alpha\beta}$, each $\Am^\alpha$ is skew-symmetric, and
\[\bm{\zeta} = \sum_{\alpha=1}^{s} \mu_\alpha\, \Am^\alpha\otimes \Am^\alpha.\]
\end{remark}


Proposition~\ref{prop:tensorcomp} is constructive and could be used as a recipe for computing the tensors $\bm{\xi},\bm{\zeta}$ which are necessary for the metriplectic ROM \eqref{eq:ROM}.  On the other hand, an even simpler description of these objects can be found by appealing to some basic facts from exterior algebra; unfamiliar readers will find everything necessary in e.g. \cite[Chapter 1]{hestenes2012} or \cite[Chapter 19]{tu2017}.  Beginning with the computation of $\bm{\xi}$, recall that any SS matrix (more formally, any antisymmetric $(0,2)$-tensor) decomposes as a sum of basis bivectors $\ee_i\wedge\ee_j = \ee_i\otimes\ee_j-\ee_j\otimes\ee_i.$  In view of this, it is convenient to identify the matrix $\LL$ with the bivector sum $\LL = \sum_{j<i}L^{ij}\ee_i\wedge\ee_j$, where the functions $L^{ij}(\xx)$ may depend on the state $\xx$.  The action of $\LL$ on a vector $\bb{v} \in \mathbb{R}^N$ given by right contraction $\LL\cdot\bb{v}$ is then identical to the matrix-vector product, since
\[ \bb{Lv} = \sum_{j<i}\left(L^{ij}\left(\ee_j\cdot\bb{v}\right)\ee_i - L^{ij}\left(\ee_i\cdot\bb{v}\right)\ee_j\right) = \sum_{j<i}L^{ij}v_j\ee_i - \sum_{j>i} L^{ji}v_j\ee_i = \sum_{i,j}L^{ij}v_j\ee_i, \]
where the skew-symmetry of $\LL$ was used along with the relationship 
\begin{equation}\label{eq:leftfromright}
   \left(\bb{b}\wedge\bb{c}\right)\cdot\bb{a} = (-1)^{1(2+1)}\bb{a}\cdot\left(\bb{b}\wedge\bb{c}\right) = (\bb{a}\cdot\bb{c})\bb{b}-(\bb{a}\cdot\bb{b})\bb{c}, \qquad \bb{a},\bb{b},\bb{c}\in\mathbb{R}^N.
\end{equation}
The advantage of this identification is a simple and coordinate-free solution to \eqref{eq:tens} which is computable following Proposition~\ref{prop:tensorcomp} and does not require the storage of a rank-3 tensor. Before continuing, note the following general relationship for vectors $\bb{a},\bb{b}$ and  bivectors $\bb{C}$ which will be used in what follows,
\begin{equation}\label{eq:vectrivec}
    \lr{\bb{C}\wedge\bb{b}}\cdot\bb{a} = \bb{a}\cdot\lr{\bb{C}\wedge\bb{b}} = \lr{\bb{a}\cdot\bb{C}}\wedge\bb{b} + \lr{\bb{a}\cdot\bb{b}}\bb{C}.
\end{equation}
Now, choose an index $k_1$ such that $S^{k_1}\neq 0$ and define 
\[\bm{\xi} = \LL\wedge\sk,\qquad \sk=\frac{\ee_{k_1}}{S^{k_1}}.\]
Using \eqref{eq:vectrivec} and the notational convenience $\bm{\xi}\lr{\vv} = \bm{\xi}\cdot\vv$ for vectors $\vv\in\mathbb{R}^N$, it follows from the definitions of $\LL,\sk$, the antisymmetry of the exterior product, and the degeneracy condition $\LL\nabla S = 0$ that $\bm{\xi}$ is totally antisymmetric and satisfies the equality
\[ \bm{\xi}\left(\nabla S\right) = \left(\nabla S \cdot \LL\right)\wedge\sk + \lr{\sk\cdot\nabla S}\LL = -\LL\nabla S\wedge\sk + \LL = \LL, \]
so that $\bm{\xi}$ also solves \eqref{eq:tens}.  Therefore, it is straightforward to verify that the reduced tensor $\xih$ which solves the normal equations \eqref{eq:normal} is 
\[ \xih = \ut\LL\uu\wedge\ut\sk = \Lb\wedge \hsk, \]
where matrix products occur before the wedge product by convention, $\Lb = \ut\LL\uu$ as before, and $\hsk = \ut\sk$.  Using again  \eqref{eq:vectrivec} along with the skew-symmetry of $\LL$, this implies that the structure-preserving counterpart to $\LL$ is given in bivector form by 
\[ \Lh = \xih\left(\nsh\right) = -\Lb\nsh\wedge\hsk + \lr{\hsk\cdot\nsh}\Lb.\]
The degeneracy condition $\Lh\nsh = 0$ is now satisfied by construction, since $\LL$ is skew-symmetric and so 
\[ \Lh\nsh = \nsh\cdot\Lb\lr{\nsh} = (0)\hsk +  \lr{\hsk\cdot\nsh}\Lb\nsh - \lr{\hsk\cdot\nsh}\Lb\nsh = \bm{0}, \]
which follows from the same algebraic relationships.

\begin{remark}
As in the computation above, it is usually beneficial to avoid expressing multivectors in a traditional tensor basis, since the number of terms grows factorially with the degree.  While bivectors are easily expressed coordinate-wise using the relations $\ee_i\wedge\ee_j = \ee_i\otimes\ee_j-\ee_j\otimes\ee_i$, basis trivectors already require a 6-term sum indexed over the permutation group $S_3$, producing much larger intermediates which must be computed and stored.
\end{remark}

It is similarly beneficial to rewrite the 4-tensor $\bm{\zeta}$.  Note that the expression for $\bm{\zeta}$ in Proposition~\ref{prop:tensorcomp} is already somewhat useful for this purpose because it enables the computation of the metriplectic ROM \eqref{eq:ROM} without actually storing or manipulating the full object.  This is because the term entering \eqref{eq:FOM} rewrites as
\[ \MM\nabla S = \bm{\zeta}\left(\nabla E,\nabla E\right)\nabla S =  \sum_{\alpha=1}^r \lambda_\alpha \left(\nabla S \cdot \Am^\alpha\nabla E\right)\Am^\alpha\nabla E, \]
so that $\bm{\zeta}$ is accessed exclusively through the matrices $\Am^\alpha$.  However, following the construction in Proposition~\ref{prop:tensorcomp} still requires storing an array of skew-symmetric matrices which are extraordinarily sparse and often depend on the state $\xx$, adding unnecessary computational expense.  While the $\xx$-dependence is not easy to handle, the storage and computation of the $\Am^\alpha$ can be reduced with an exterior algebraic factorization as before.

\begin{remark}
In practice, it is usually easier to work with $\bb{m}^\alpha \gets \sqrt{\lambda_\alpha}\bb{m}^\alpha$.  From now on, we write $\MM = \sum_\alpha \bb{m}^\alpha\otimes\bb{m}^\alpha$ with the understanding that the eigenvectors $\bb{m}^\alpha$ no longer have unit magnitude.  The reader can check that the conclusions of  Proposition~\ref{prop:tensorcomp} are unaffected by this change.  
\end{remark}

Again, it is useful to identify the matrices $\Am^\alpha$ with the bivectors $\Am^\alpha = \sum_{j<i}A^{\alpha,ij}\ee_i\wedge\ee_j$, so that for each $1\leq\alpha\leq r$ there is a decomposition following Proposition~\ref{prop:tensorcomp},
\[ \Am^\alpha = \am^\alpha\wedge\ee_{k_0}, \qquad \am^\alpha = \frac{\bb{m}^\alpha}{E^{k_0}}, \]
where $k_0$ is an index such that $E^{k_0}\neq 0$.  Since the normal equations \eqref{eq:normal} for $\zh$ are solved when $\hat{\Am}^\alpha = \ut\Am^\alpha\uu$, this means the reduced-order $\hat{\Am}^\alpha$ are given by 
\[ \hat{\Am}^\alpha = \ah^\alpha \wedge \uu^{k_0},\qquad \ah^\alpha = \frac{\ut\bb{m}^\alpha}{E^{k_0}}, \]
where $\uu^{k_0} = \ut\ee_{k_0}$ denotes the $k_0^{th}$ row of $\uu$.  This affords a simple representation for the matrix-vector products
\[ \hat{\Am}^\alpha\neh = -\neh\cdot\lr{\ah^\alpha\wedge\uu^{k_0}} = \lr{\neh\cdot\uu^{k_0}}\ah^\alpha - \lr{\neh\cdot\ah^\alpha}\uu^{k_0},\]
leading to greater efficiency in the numerical implementation.  Moreover, it is easy to see that $\hat{\Am}^\alpha$ remains skew-symmetric for all $\alpha$, so that $\zh\lr{\neh,\neh}\neh = \bm{0}$ by construction.  Putting all of the computations in this Section together yields the following expression of the ROM \eqref{eq:ROM} which facilitates the numerical experiments in Section~\ref{sec:numerics}.

\begin{theorem}[Metriplectic structure-preserving ROM]\label{thm:ROM}
    Suppose $\LL$, $\MM = \sum_\alpha\bb{m}^\alpha\otimes\bb{m}^\alpha$, $E$ and $S$ describe a metriplectic dynamical system \eqref{eq:FOM} with state $\xx\in\mathbb{R}^N$ and associated tensors $\bm{\xi}, \bm{\zeta}$ satisfying \eqref{eq:tens} and \eqref{eq:syms} for all $\xx$.  Suppose $\xh\in\mathbb{R}^n$ is a low-dimensional approximation to the state in the sense that $\xx\approx\xt = \xx_0 + \uu\xh$ for some initial state $\xx_0\in\mathbb{R}^N$ and linear mapping $\uu\in\mathbb{R}^{N\times n}$.  Let $1\leq k_0,k_1\leq N$ be indices such that $E^{k_0}(\xx)\neq 0$ and $S^{k_1}(\xx)\neq 0$ for all $\xx$ (otherwise the conclusion holds locally) and define 
    \begin{align*}
        \xih &= \Lb\wedge\hsk,\qquad \hsk=\frac{\uu^{k_1}}{S^{k_1}}, \\
        \hat{\Am}^\alpha &= \ah^\alpha \wedge \uu^{k_0},\qquad \ah^\alpha = \frac{\ut\bb{m}^\alpha}{E^{k_0}},
    \end{align*}
    where $\Lb = \ut\LL\uu$ and $\zh = \sum_\alpha \hat{\Am}^\alpha\otimes\hat{\Am}^\alpha$.  Then, denoting $\hat{F} = F\circ\xt$ for any real-valued function $F$, the ROM 
    \begin{align*}
        \dot{\xh} &= \rb{\xh}{\neh\lr{\xh}} + \ib{\xh}{\nsh\lr{\xh}} \\
        &:=\Lh\lr{\xt}\neh\lr{\xh} + \Mh\lr{\xt}\nsh\lr{\xh} \\
        &= \xih\lr{\xt}\lr{\nsh\lr{\xh}}\neh\lr{\xh} + \zh\lr{\xt}\lr{\neh\lr{\xh},\neh\lr{\xh}}\nsh\lr{\xh},
    \end{align*}
    preserves the original metriplectic structure.  Suppressing the potential state dependence, this dynamical system has the explicit representation
    \begin{align*}
        \dot{\xh} &= \lr{\neh\cdot\Lb\nsh}\hsk - \lr{\hsk\cdot\neh}\Lb\nsh + \lr{\hsk\cdot\nsh}\Lb\neh \\
        &\quad+ \sum_{\alpha=1}^r \lr{\Am^\alpha\neh\cdot\nsh}\Am^\alpha\neh,
    \end{align*}
    where $\Am^\alpha\neh$ is expressed in terms of $\ah$ and the $k_0^{th}$ row of $\uu$ as
    \[ \hat{\Am}^\alpha\neh = \lr{\neh\cdot\uu^{k_0}}\ah^\alpha - \lr{\neh\cdot\ah^\alpha}\uu^{k_0}.\]
\end{theorem}

\begin{remark}
Note that the quantities $\Lh,\Am^\alpha,\bb{m}^\alpha,\hsk,\ah^\alpha$ may depend on $\xt$, so that the metriplectic ROM generally depends on the full $N$-dimensional input space.  On the other hand, there are many special cases where this dependence can be mitigated or removed entirely.  Example \ref{ex:big} in Section 5 gives one such illustration.  Future work will also consider hyper-reduction strategies such as DEIM \cite{chaturantabut2010} for this purpose.
\end{remark}

\section{Error estimate}
It is important to know that the metriplectic ROM described in Theorem~\ref{thm:ROM} will converge to the true solution as the reduced dimension $n$ approaches the full resolution $N$.  The results of this Section show that this is indeed the case when the mapping $\uu$ is generated by POD and the snapshot matrix is augmented with gradient information.  To explain why, let $\norm{\cdot}$ denote the usual norm in $L^2([0,T])$ and denote $\ww = \xx-\xx_0$.  Choose weighting parameters $\mu,\nu\in\mathbb{R}$ as well as $n_t$ instants $t_i\in[0,T]$ and consider the snapshot matrix $\bb{Y}\in\mathbb{R}^{N\times 3n_t}$,
\[\bb{Y} = \begin{pmatrix}\ww\lr{t_i} & \mu\nabla E\lr{\xx(t_i)} & \nu\nabla S\lr{\xx(t_i)}\end{pmatrix}_{i=1}^{n_t}.\]

\begin{remark}
For simplicity, all experiments in Section~\ref{sec:numerics} use the weights $\mu = \nu = 1$.
\end{remark}

Suppose $\bb{Y}$ has rank $r$ and denote by $\sigma_i$ the $i^{th}$  singular value of $Y$.  Then, if $\uu \in \mathbb{R}^{N\times n}$ is the rank $n<r$ matrix of left singular vectors and $\bb{P}^\perp = \bb{I}-\uu\ut$ is orthogonal projection onto the complement of $\mathrm{Span}\lr{\uu}$, standard POD theory (see e.g. \cite{liang2002}) implies that 
\begin{equation}\label{eq:PODest}
    \norm{\bb{P}^\perp\ww}^2 + \mu^2\norm{\bb{P}^\perp\nabla E}^2 + \nu^2\norm{\bb{P}^\perp\nabla S}^2 = \sum_{j>n} \sigma_j^2.
\end{equation}
Recall also the Lipschitz constant and logarithmic Lipschitz constant, denoted for a general function $F$ between metric spaces as 
\begin{align*}
    C_F = \sup_{\bb{u}\neq\bb{v}} \frac{\nn{F(\bb{u})-F(\bb{v})}}{\nn{\bb{u}-\bb{v}}}, \qquad c_F = \sup_{\bb{u}\neq\bb{v}} \frac{\lr{\bb{u}-\bb{v}}\cdot\lr{F(\bb{u})-F(\bb{v})}}{\nn{\bb{u}-\bb{v}}^2}.
\end{align*}
This can be used to derive the following result for the present scheme.


\begin{theorem}\label{thm:conv}
Let $\xx(t)$ denote the solution to the FOM \eqref{eq:FOM} with initial condition $\xx_0$ and let $\xh(t)$ denote the solution to the ROM \eqref{eq:ROM} with $\xih = \ut\bm{\xi}\lr{\uu}\uu$, $\zh = \ut\bm{\zeta}\lr{\uu,\uu}\uu$ and initial condition $\xh_0 = \bm{0}$.  Suppose $\bm{\xi},\bm{\zeta},\nabla E,\nabla S$ are Lipschitz continuous with bounded norms in space as well as  $L^2$-integrable in time.  Then, the approximation error satisfies
\[ \norm{\xx-\left(\xx_0+\uu\xh\right)}^2 \leq C(T,\bm{\xi},\bm{\zeta},\mu,\nu)\sum_{j>n}\sigma^2_j. \]
\end{theorem}
\begin{proof}
First, suppose $\xx_0 = \bm{0}$, so that $\xt = \uu\xh$ and the approximation error becomes
\[ \xx - \uu\xh = \left(\xx-\uu\ut\xx\right)+\left(\uu\ut\xx-\uu\xh\right) = \bb{P}^\perp\xx+\yy.\]
For cleanliness of notation, we denote $\bar{\bb{X}} = \uu\ut\bb{X}$ and
\begin{align*}
    \bb{F}(\xx) &= \bar{\bm{\xi}}(\xx)\left(\bar{\nabla}S(\xx)\right) \bar{\nabla}E(\xx) + \bar{\bm{\zeta}}(\xx)\lr{\bar{\nabla}E(\xx),\bar{\nabla}E(\xx)}\bar{\nabla}S(\xx).
\end{align*}
It follows from \eqref{eq:FOM} and \eqref{eq:ROM} that the component of the error captured by $\yy$ decomposes into three terms,
\begin{align*}
    \dot{\yy} &= \dot{\bar{\xx}}-\uu\dot{\xh} = \bar{\bm{\xi}}\lr{\nabla S(\xx)}\nabla E(\xx) + \bar{\bm{\zeta}}\lr{\nabla E(\xx),\nabla E(\xx)}\nabla S(\xx) -\bb{F}\lr{\uu\xh} \\
    &= \bar{\bm{\xi}}\lr{\nabla S(\xx)}\nabla E(\xx) + \bar{\bm{\zeta}}\lr{\nabla E(\xx),\nabla E(\xx)}\nabla S(\xx) - \bb{F}(\xx) \\
    &\qquad+ \left(\bb{F}(\xx)-\bb{F}\lr{\bar{\xx}}\right) +  \left(\bb{F}\lr{\bar{\xx}}-\bb{F}\lr{\uu\xh}\right) \\
    &:= \bb{T}_1 + \bb{T}_2 + \bb{T}_3.
\end{align*}
Suppressing the $\xx$ argument, $\bb{T}_1$ can be written as
\begin{align*}
    \bb{T}_1 &= \bar{\bm{\xi}}\lr{\bb{P}^\perp\nabla S}\nabla E + \bar{\bm{\xi}}\lr{\bar{\nabla}S}\bb{P}^\perp\nabla E \\
    &\qquad+ \bar{\bm{\zeta}}\lr{\bb{P}^\perp\nabla E,\nabla E}\nabla S + \bar{\bm{\zeta}}\lr{\bar{\nabla}E, \bb{P}^\perp\nabla E}\bar{\nabla}S + \bar{\bm{\zeta}}\lr{\bar{\nabla}E,\bar{\nabla}E}\bb{P}^\perp\nabla S,
\end{align*}
so that its norm can be estimated as 
\begin{align*}
    \nn{\bb{T}_1} &\leq \lr{\nn{\bar{\bm{\xi}}}\nn{\bar{\nabla}S} + \nn{\bar{\bm{\zeta}}}\nn{\nabla E}\nn{\nabla S} + \nn{\bar{\bm{\zeta}}}\nn{\bar{\nabla}E}\nn{\bar{\nabla} S}}\nn{\bb{P}^\perp\nabla E} \\
    &\qquad+ \lr{\nn{\bar{\bm{\xi}}}\nn{\nabla E}+\nn{\bar{\bm{\zeta}}}\nn{\bar{\nabla}E}^2}\nn{\bb{P}^\perp\nabla S} \\
    &:= f_1\nn{\bb{P}^\perp\nabla E} + f_2\nn{\bb{P}^\perp\nabla S}.
\end{align*}
The term $\bb{T}_2$ can be rewritten in a similar fashion.  Collecting all terms involving $\bm{\xi}$ yields
\begin{align*}
    &\bar{\bm{\xi}}(\xx)\lr{\bar{\nabla}S(\xx)}\bar{\nabla}E(\xx) - \bar{\bm{\xi}}\lr{\bar{\xx}}\lr{\bar{\nabla}S\lr{\bar{\xx}}}\bar{\nabla}E\lr{\bar{\xx}} \\
    &\qquad= \lr{\bar{\bm{\xi}}(\xx)-\bar{\bm{\xi}}\lr{\bar{\xx}}}\lr{\bar{\nabla}S(\xx)}\bar{\nabla}E(\xx) \\
    &\qquad\qquad+ \bar{\bm{\xi}}\lr{\bar{\xx}}\lr{\bar{\nabla}S(\xx)-\bar{\nabla}S\lr{\bar{\xx}}}\bar{\nabla}E(\xx) \\
    &\qquad\qquad+ \bar{\bm{\xi}}\lr{\bar{\xx}}\lr{\bar{\nabla}S\lr{\bar{\xx}}}\lr{\bar{\nabla}E(\xx)-\bar{\nabla}E\lr{\bar{\xx}}},
\end{align*}
while collecting the terms involving $\bm{\zeta}$ gives
\begin{align*}
    &\bar{\bm{\zeta}}(\xx)\lr{\bar{\nabla}E(\xx),\bar{\nabla}E(\xx)}\bar{\nabla}S(\xx) - \bar{\bm{\zeta}}\lr{\bar{\xx}}\lr{\bar{\nabla}E\lr{\bar{\xx}},\bar{\nabla}E\lr{\bar{\xx}}}\bar{\nabla}S\lr{\bar{\xx}} \\
    &\qquad= \lr{\bar{\bm{\zeta}}(\xx)-\bar{\bm{\zeta}}\lr{\bar{\xx}}}\lr{\bar{\nabla}E(\xx),\bar{\nabla}E(\xx)}\bar{\nabla}S(\xx) \\
    &\qquad\qquad+ \bar{\bm{\zeta}}\lr{\bar{\xx}}\lr{\bar{\nabla}E(\xx)-\bar{\nabla}E\lr{\bar{\xx}},\bar{\nabla}E(\xx)}\bar{\nabla}S(\xx) \\
    &\qquad\qquad+ \bar{\bm{\zeta}}\lr{\bar{\xx}}\lr{\bar{\nabla} E\lr{\bar{\xx}},\bar{\nabla}E(\xx)-\bar{\nabla} E\lr{\bar{\xx}}}\bar{\nabla}S(\xx) \\
    &\qquad\qquad+ \bar{\bm{\zeta}}\lr{\bar{\xx}}\lr{\bar{\nabla}E\lr{\bar{\xx}},\bar{\nabla}E\lr{\bar{\xx}}}\lr{\bar{\nabla}S(\xx)-\bar{\nabla}S\lr{\bar{\xx}}}.
\end{align*}
Hence, the norm of $\nn{\bb{T}_2}$ is bounded above by
\begin{align*}
    \nn{\bb{T}_2} &\leq \nn{\uu\ut}\nn{\bar{\nabla}S(\xx)}\nn{\bar{\nabla}E(\xx)}\lr{C_\xi + C_{\zeta}\nn{\bar{\nabla}E(\xx)}}\nn{\bb{P}^\perp\xx} \\
    &\qquad+ \nn{\uu\ut}\nn{\bar{\bm{\xi}}\lr{\bar{\xx}}}\lr{C_{\nabla S}\nn{\bar{\nabla}E(\xx)}+C_{\nabla E}\nn{\bar{\nabla}S\lr{\bar{\xx}}}}\nn{\bb{P}^\perp\xx} \\
    &\qquad+ C_{\nabla E}\nn{\uu\ut}\nn{\bar{\bm{\zeta}}\lr{\bar{\xx}}}\nn{\bar{\nabla}S(\xx)}\lr{\nn{\bar{\nabla}E(\xx)}+\nn{\bar{\nabla}E\lr{\bar{\xx}}}}\nn{\bb{P}^\perp\xx} \\
    &\qquad+ C_{\nabla S}\nn{\uu\ut}\nn{\bar{\bm{\zeta}}\lr{\bar{\xx}}}\nn{\bar{\nabla}E\lr{\bar{\xx}}}^2\nn{\bb{P}^\perp\xx} \\
    &:= f_3\nn{\bb{P}^\perp\xx}
\end{align*}
The assumptions of the Theorem imply that the supremum of $f_3$ over $\xx\neq\bar{\xx}$ is finite, so the above inequality combined with $\bb{T}_2 = \bb{F}\lr{\xx}-\bb{F}\lr{\bar{\xx}}$ implies that $C_F \leq \sup f_3 < \infty$ and the Lipschitz constant $C_F$ exists.  As $|c_F| \leq C_F$, it follows that the logarithmic Lipschitz constant $c_F$ exists also.


Therefore, testing each term of $\dot{\yy}$ against $\yy$ and estimating with Cauchy-Schwarz yields
\begin{align*}
    \yy\cdot\bb{T}_1 &\leq \left(f_1\nn{\bb{P}^\perp\nabla E}+f_2\nn{\bb{P}^\perp\nabla S}\right)\nn{\yy}, \\
    \yy\cdot\bb{T}_2 &\leq f_3\nn{\bb{P}^\perp\xx}\nn{\yy}, \\
    \yy\cdot\bb{T}_3 &= \frac{\yy\cdot\bb{T}_3}{\nn{\yy}^2}\nn{\yy}^2 \leq c_F\nn{\yy}^2,
\end{align*}
where the scalar functions $f_j$ may depend on $t$ since their terms may depend on $\xx$.  Noting that $\dot{\nn{\yy}} = \left(1/\nn{\yy}\right)\lr{\yy\cdot\dot{\yy}}$ and using the estimates above yields the inequality
\[ \dot{\nn{\yy}}-c_F\nn{\yy} \leq f_1\nn{\bb{P}^\perp\nabla E} + f_2\nn{\bb{P}^\perp\nabla S} + f_3\nn{\bb{P}^\perp\xx}. \]
Multiplying both sides by the integrating factor $e^{-c_Ft}$ and applying Gronwall's inequality for $t\in[0,T]$ then gives
\[\nn{\yy} \leq \int_0^t e^{c_F(t-\tau)}\left(f_1\nn{\bb{P}^\perp\nabla E}+f_2\nn{\bb{P}^\perp\nabla S}+f_3\nn{\bb{P}^\perp\xx}\right)\,d\tau,\]
where we have used that $\yy(0) = \bm{0}$.
Then, Cauchy-Schwarz along with $T\geq t$ imply
\begin{align*}
    \nn{\yy}^2 &\leq C_4 \norm{f_1\nn{\bb{P}^\perp\nabla E}+f_2\nn{\bb{P}^\perp\nabla S}+f_3\nn{\bb{P}^\perp\xx}}^2 \\
    &\leq C_4\lr{\|f_1\|^2\norm{\bb{P}^\perp\nabla E}^2+\|f_2\|^2\norm{\bb{P}^\perp\nabla S}^2+\|f_3\|^2\norm{\bb{P}^\perp\xx}^2},
\end{align*}
where $C_4 = \int_0^T e^{2c_F(T-\tau)} = (e^{2c_FT}-1)/2c_F$.  Finally, integrating both sides in $t$ yields
\[\norm{\yy}^2 \leq TC_4\left(\|f_1\|^2\norm{\bb{P}^\perp\nabla E}^2+\|f_2\|^2\norm{\bb{P}^\perp\nabla S}^2+\|f_3\|^2\norm{\bb{P}^\perp\xx}^2\right),\]
so that in view of \eqref{eq:PODest} the error can be estimated as
\begin{align*}
    \norm{\xx-\uu\xh}^2 &\leq \norm{\bb{P}^\perp\xx}^2+\norm{\yy}^2 \\
    &\leq \lr{1 + TC_4\lr{\frac{\|f_1\|^2}{\mu^2} + \frac{\|f_2\|^2}{\nu^2}+\|f_3\|^2}}\sum_{j>n}\sigma^2_j = C\sum_{j>n}\sigma^2_j,
\end{align*}
as desired.  To complete the argument, consider what happens when $\xx_0\neq \bm{0}$. In this case, $\bb{w}=\xx-\xx_0$ satisfies $\dot{\bb{w}}=\dot{\bb{x}}$ and the ROM error is
\[ \xx - \left(\xx_0+\uu\xh\right) = \lr{\bb{w}-\bar{\bb{w}}}+\lr{\bar{\bb{w}}-\uu\xh} = \bb{P}^\perp\bb{w}+\bb{z}. \]
Since $\norm{\bb{P}^\perp\bb{w}}$ is controlled by \eqref{eq:PODest} and $\bb{w}_0 = \bm{0}$, this case can be analyzed identically to the case $\xx_0=\bm{0}$.
\end{proof}

\section{Numerical Examples}\label{sec:numerics}
This section evaluates the performance of the metriplectic ROM (referred to as the SP-ROM) proposed in Theorem~\ref{thm:ROM} on some benchmark problems.  To create a fair comparison, its performance is measured against that of the standard Galerkin POD-ROM \eqref{eq:nospROM} referred to as the G-ROM and a mild extension \eqref{eq:oldrom} of the Hamiltonian POD-ROM from \cite{gong2017} referred to as the EH-ROM which approximately preserves the compatibility conditions \eqref{eq:compatibility}.  The error metrics used for this purpose are the relative $\ell^2$ error and the maximum $\ell^2$ error, denoted respectively as
\[ \mathcal{E}_r\lr{\xt} = \sqrt{\frac{\sum_i\nn{\xx(t_i)-\xt(t_i)}^2}{\sum_i\nn{\xx(t_i)}^2}}, \qquad \mathcal{E}_\infty\lr{\xt} = \max_i \nn{\xx(t_i)-\xt(t_i)}, \]
where $\xx$ is the true solution and $1\leq i \leq n_t$ are the indices of the discretization points $t_i\in[0,T]$.  Besides the error metrics, the energy difference $\nn{E(T)-E_0}$ ($E_0 = E(t_0)$) is reported as well as the online computational time in seconds necessary for integrating each model.  This collection of data provides a rough measure of model quality which is used to draw conclusions about ROM performance.  The experiments chosen are a low-dimensional example motivated by gas kinetics and an infinite-dimensional example coming from elasticity theory.  In both cases, the initial conditions are parameterized and used to train the mapping $\uu$, and the ROMs are used to predict unseen solutions given relevant initial data.  In all cases, the SciPy \cite{scipy2020} interface to LSODA \cite{hindmarsh1983} is used to integrate the resulting ODE systems with an error tolerance of $10^{-14}$.  All experiments are carried out on a 2022 M1 MacBook Pro with 32GB of RAM.

\begin{remark}
As computational efficiency is not emphasized in this work, all ROMs in these examples are implemented sub-optimally compared to the FOM, which interfaces directly with optimized SciPy routines.  This causes the ROMs to appear slower relative to the FOM than they would likely be in practice, and occasionally even slower than the FOM itself.  
\end{remark}

\subsection{Two gas containers exchanging heat and volume}
To support the claims proven in Theorem~\ref{thm:conv} and Theorem~\ref{thm:ROM}, it is useful to consider the following low-dimensional example from \cite{shang2020}.  Two gas containers are allowed to exchange heat and volume on either side of a wall. The state variable is $\xx = \begin{pmatrix} q & p & S_1 & S_2 \end{pmatrix}^\intercal \in \mathbb{R}^4$ where $q,p$ are the position resp. momentum of the wall and $S_1,S_2$ are the entropy of the gases in the respective containers.  The entropy function is then $S = S_1 + S_2$ and the energy function is 
\[ E(\xx) = \frac{p^2}{2m} + E_1 + E_2 := \frac{p^2}{2m} + \left(\frac{e^\frac{S_1}{(N k_B)}}{\hat{c}\,q}\right)^\frac{2}{3} + \left(\frac{e^\frac{S_2}{(N k_B)}}{\hat{c}\,(2-q)}\right)^\frac{2}{3}, \quad \hat{c} = \left(\frac{4\pi m^2}{3h^2N}\right)^\frac{3}{2}\frac{e^\frac{5}{2}}{N},\]
where $m$ is the mass of the wall, $N$ is the number of gas particles, $h$ is the Planck constant and $k_B$ is the Boltzmann constant.  For the rest of the discussion, a normalization is assumed such that $m=Nk_B=\hat{c}=1$.  This system then has the metriplectic form
\begin{align*}
    \begin{pmatrix}
    \dot{q} \\ \dot{p} \\ \dot{S}_1 \\ \dot{S}_2
    \end{pmatrix}
    = 
    \begin{pmatrix}
    p \\
    \frac{2}{3}\left(\frac{E_1}{q}-\frac{E_2}{2-q}\right) \\ \frac{\gamma}{T_1}\left(\frac{1}{T_1}-\frac{1}{T_2}\right) \\ \frac{-\gamma}{T_2}\left(\frac{1}{T_1}-\frac{1}{T_2}\right)
    \end{pmatrix}
    = \LL\nabla E(\xx) + \MM(\xx)\nabla S,
\end{align*}
where $T_i = \partial_{S_i}E_i = (2/3)E_i$, $\nabla S = \begin{pmatrix} 0 & 0 & 1 & 1
\end{pmatrix}^\top$, $\gamma$ regulates the degree of heat transfer, and 
\begin{equation*}
    \LL = 
    \begin{pmatrix}
    0 & 1 & 0 & 0 \\
    -1 & 0 & 0 & 0 \\
    0 & 0 & 0 & 0 \\
    0 & 0 & 0 & 0
    \end{pmatrix},
    \quad \MM = \gamma
    \begin{pmatrix}
    0 & 0 & 0 & 0 \\
    0 & 0 & 0 & 0 \\
    0 & 0 & T_1^{-2} & -(T_1T_2)^{-1} \\
    0 & 0 & -(T_1T_2)^{-1} & T_2^{-2},
    \end{pmatrix},
\end{equation*}
\begin{equation*}
   \nabla E = \frac{2}{3}
    \begin{pmatrix}
    -\left(\frac{e^{2S_1}}{q^5}\right)^{\frac{1}{3}} + \left(\frac{e^{2S_2}}{(2-q)^5}\right)^{\frac{1}{3}} \\
    (3/2)\,p \\
    E_1 \\
    E_2
    \end{pmatrix}.
\end{equation*}

\begin{figure}
    \centering
    \includegraphics[width=\textwidth]{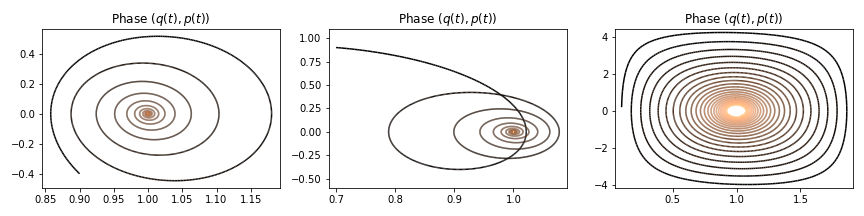}
    \caption{The phase portraits of three qualitatively different solutions to the gas container problem with corresponding initial conditions $\begin{pmatrix}0.9&-0.4&2.4&2.0\end{pmatrix}^\intercal$ (left), $\begin{pmatrix}0.7&0.9&1.1&2.9\end{pmatrix}^\intercal$ (middle), and $\begin{pmatrix}0.1&0.2&1.6&1.8\end{pmatrix}^\intercal$ (right).}
    \label{fig:gasFOM}
\end{figure}


Note that the matrix $\MM$ decomposes (nonuniquely) as
\[\MM=\bb{m}\otimes\bb{m},\qquad \bb{m} = \sqrt{\gamma}\begin{pmatrix} 0 & 0 & T_1^{-1} & -T_2^{-1}\end{pmatrix}^\intercal,\] 
so that by choosing indices $k_0,k_1$ such that $E^{k_0}\neq 0$ and $S^{k_1}\neq 0$ the tensors $\bm{\zeta},\bm{\xi}$ can be computed following Proposition~\ref{prop:tensorcomp}.  The choice $k_0=k_1=3$ yields the reduced-order objects
\begin{align*}
    \xih = \Lb\wedge\uu^{3}, \qquad \zh\lr{\xt} = \hat{\Am}\lr{\xt}\otimes\hat{\Am}\lr{\xt}, \qquad \hat{\Am}\lr{\xt} = \frac{3}{2}\frac{\ut\bb{m}\lr{\xt}}{E_1\lr{\xt}} \wedge \uu^3,
\end{align*}
making it clear that $\xih$ can be precomputed while some components of $\zh$ must be computed online.  For the present experiment, ROM performance is compared at various levels of compression when integrated over two different lengths of time, one of which extends away from the training regime.  The initial state $\xx_0$ is assumed to lie in $[0.08,1.8]\times[-1,1]\times[1,3]\times[1,3]$, and snapshots from 25 simulations with $\gamma=8$ and random, uniformly distributed initial data are used to train the POD approximation $\uu$.  Some representative solutions are displayed in Figure~\ref{fig:gasFOM}.  Snapshots of the shifted solution $\xx-\xx_0$ as well as the gradient $\nabla E\lr{\xx}$ are saved every 0.02 time instants in the interval $[0,8]$ during training and concatenated to form the snapshot matrix $\bb{Y}$.  As $\nabla S$ is a constant independent of $\bb{x}$, it is included only once as the final column of the snapshot matrix.   

The performance of each ROM is tested using the solution with corresponding initial condition $\xx_0 = \begin{pmatrix}0.9 & -0.4 & 2.4 & 2.0\end{pmatrix}^\intercal$ not included in the training set.  Table~\ref{tab:gas} compares the errors that arise when integrating to $T=8$ where snapshots end, as well as when integrating to $T=32$ which extends well past this point. Figures~\ref{fig:8gas234} and \ref{fig:32gas234} display the FOM and ROM positions $q$, entropies $S$, and energy variations $E-E_0$ over each time interval when $n=2,3,4$.  Notice that all methods converge with refinement as expected, but only the proposed SP-ROM is capable of preserving the initial energy (to order $10^{-13})$ as well as producing a reasonable entropy profile regardless of the reduced dimension.  It is also interesting to note that the simple G-ROM can produce quite unphysical results in terms of its energy and entropy profiles, including the rapid energy increase seen in Figure~\ref{fig:32gas234}.  This shows that there are clear benefits to preserving the metriplectic structure of the original system, especially in the presence of downstream procedures which rely on accurate estimates of such quantities.

\begin{table}[]\label{tab:gas}
\begin{center}
\resizebox{\columnwidth}{!}{%
\def\arraystretch{1.25}%
\begin{tabular}{|c|c|ccccc|c|c|ccccc|}
\hline
$T$ & $n$ & Method & $\mathcal{E}_r$ \% & $\mathcal{E}_\infty$ & $\nn{E(T)-E_0}$ & Time (s) & $T$ & $n$ & Method & $\mathcal{E}_r$ \% & $\mathcal{E}_\infty$ & $\nn{E(T)-E_0}$ & Time (s) \\ \hline
\multirow{10}{*}{8} & - & FOM & - & - & $5.507\times 10^{-14}$ & 0.07517 & \multirow{10}{*}{32} & - & FOM & - & - & $1.279\times 10^{-13}$ & 0.2323 \\ \cline{2-7} \cline{9-14}
& \multirow{3}{*}{2} & SP-ROM & 15.84 & 0.9166 & $1.066\times 10^{-13}$ & 0.01457 & & \multirow{3}{*}{2} & SP-ROM & 14.71 & 0.9166 & $1.243\times 10^{-13}$ & 0.02071 \\
& & EH-ROM & 16.00 & 0.9162 & 0.3127 & 0.01631 & & & EH-ROM & 15.29 & 0.9162 & $3.357\times 10^{-3}$ & 0.02199 \\
& & G-ROM & 58.96 & 2.436 & 46.29 & 0.02356 & & & G-ROM & 182.9 & 7.188 & 965.5 & 0.04956 \\ \cline{2-7} \cline{9-14}
& \multirow{3}{*}{3} & SP-ROM & 7.462 & 0.2338 & $5.861\times 10^{-14}$ & 0.07991 & & \multirow{3}{*}{3} & SP-ROM & 7.367 & 0.2338 & $1.279\times 10^{-13}$ & 0.2262 \\
& & EH-ROM & 7.303 & 0.2246 & 0.4692 & 0.06545 & & & EH-ROM & 9.050 & 0.3677 & 1.378 & 0.1906 \\
& & G-ROM & 6.808 & 0.2244 & 0.4364 & 0.06068 & & & G-ROM & 8.971 & 0.3754 & 1.422 & 0.1843 \\ \cline{2-7} \cline{9-14}
& \multirow{3}{*}{4} & SP-ROM & $3.968\times 10^{-12}$ & $3.465\times 10^{-13}$ & $4.619\times 10^{-14}$ & 0.07635 & & \multirow{3}{*}{4} & SP-ROM & $1.005\times 10^{-11}$ & $6.611\times 10^{-13}$ & $2.025\times 10^{-13}$ & 0.1734 \\
& & EH-ROM & $5.252\times 10^{-12}$ & $3.988\times 10^{-13}$ & $1.385\times 10^{-13}$ & 0.07155 & & & EH-ROM & $1.099\times 10^{-11}$ & $8.299\times 10^{-13}$ & $1.421\times 10^{-13}$ & 0.1683 \\
& & G-ROM & $6.769\times 10^{-12}$ & $3.934\times 10^{-13}$ & $2.345\times 10^{-13}$ & 0.05992 & & & G-ROM & $1.062\times 10^{-11}$ & $6.225\times 10^{-13}$ & $2.984\times 10^{-13}$ & 0.1725 \\ \hline
\end{tabular}%
}
\end{center}
\medskip
\caption{Results of the gas container experiment.  Dashes indicate ``not applicable'' to the case of the FOM.}
\end{table}

\begin{figure}
    \centering
    \includegraphics[width=\textwidth]{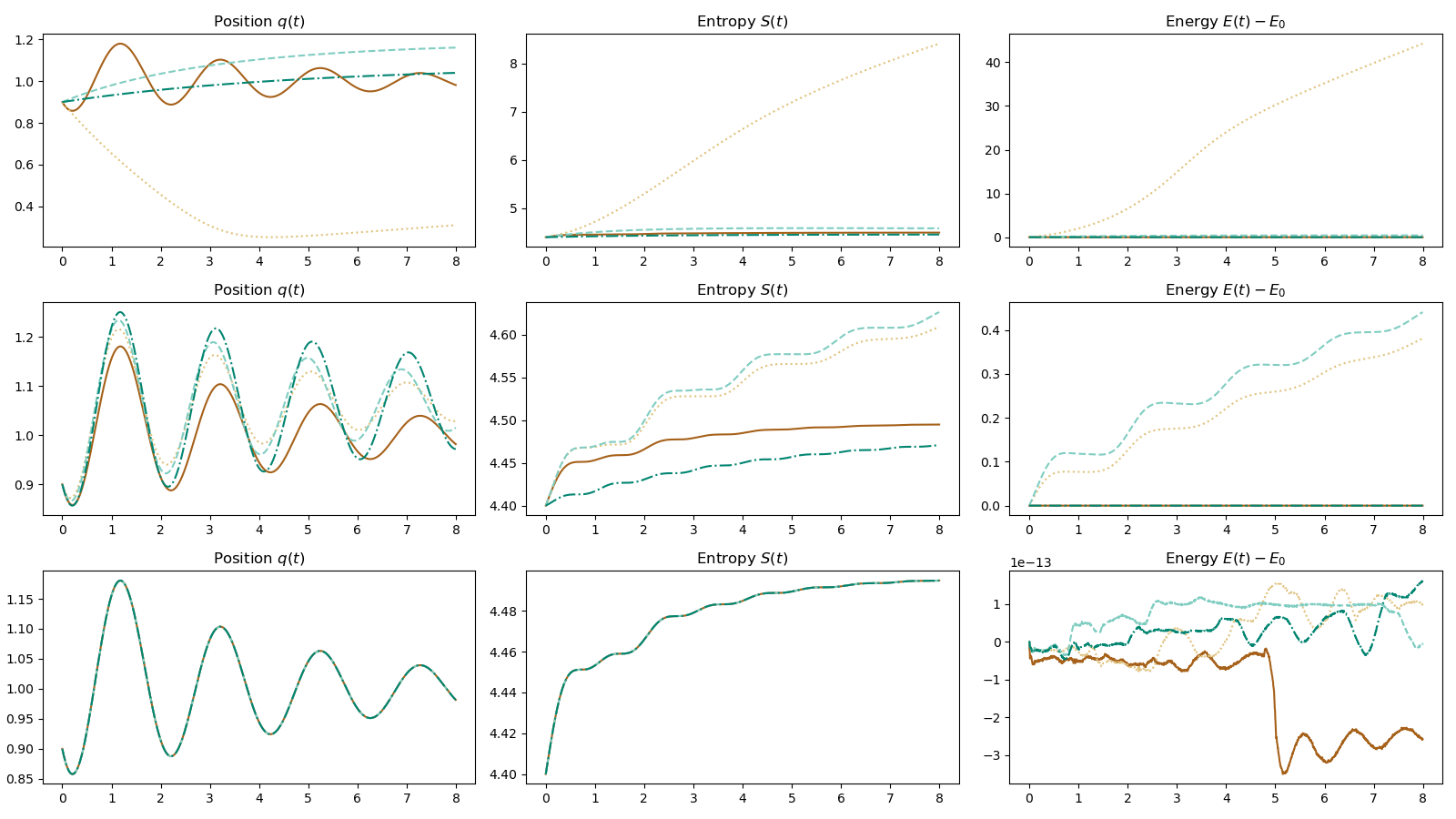}
    \caption{A comparison of ROM solutions for the 4-dimensional gas container example when $T=8$ and $n=2,3,4$, respectively.  Plotted are the \textcolor{exact}{Exact Solution (solid)}, \textcolor{no}{G-ROM (dotted)}, \textcolor{old}{EH-ROM (dashed)}, and \textcolor{new}{SP-ROM (dot-dashed)}.  Observe the convergence as established in Theorem~\ref{thm:conv}.}
    \label{fig:8gas234}
\end{figure}


\begin{figure}
    \centering
    \includegraphics[width=\textwidth]{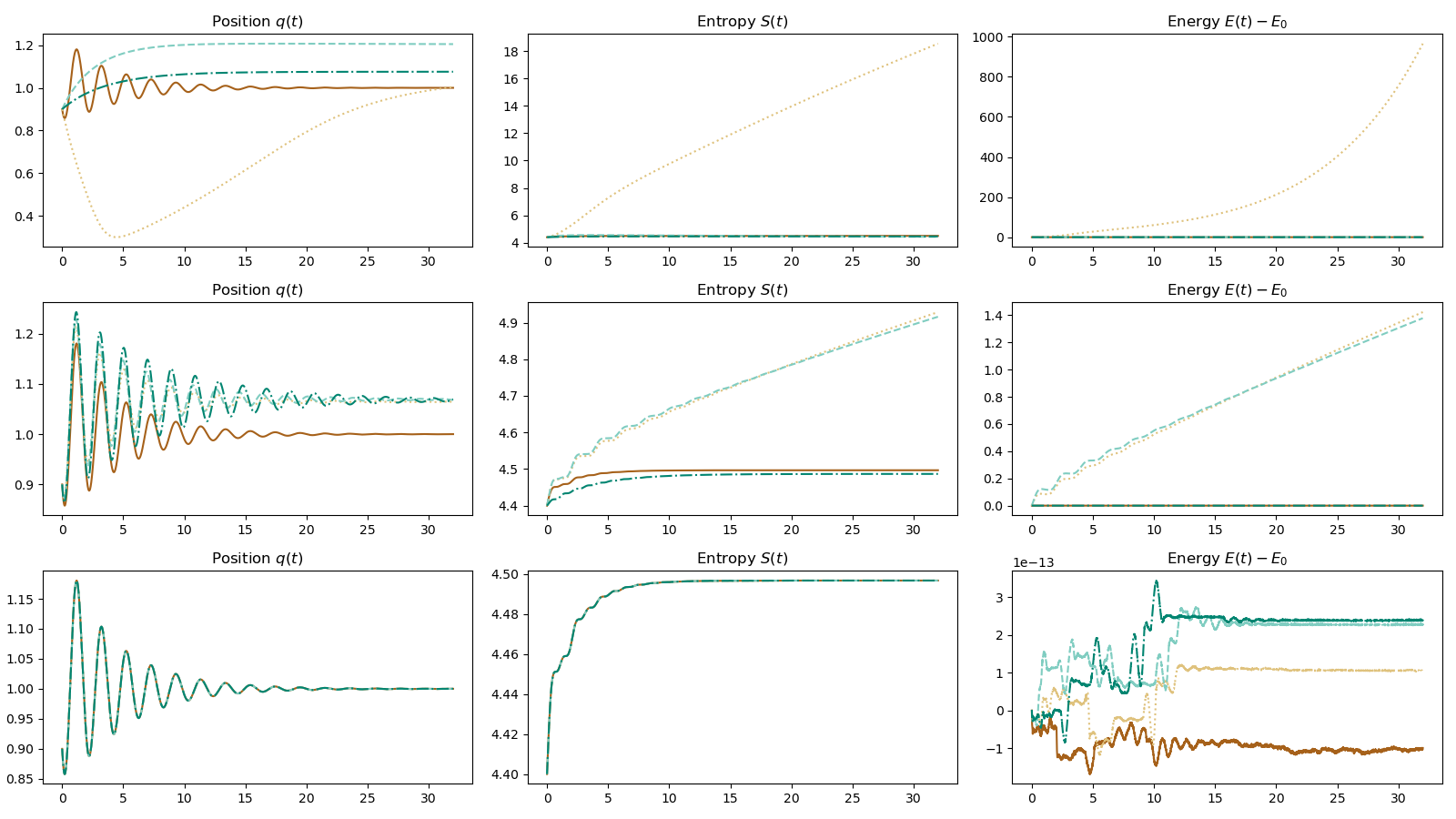}
    \caption{A comparison of ROM solutions for the 4-dimensional gas container example when $T=32$ and $n=2,3,4$, respectively.  Plotted are the \textcolor{exact}{Exact Solution (solid)}, \textcolor{no}{G-ROM (dotted)}, \textcolor{old}{EH-ROM (dashed)}, and \textcolor{new}{SP-ROM (dot-dashed)}.  Note that only the \textcolor{new}{SP-ROM} produces reasonable energy and entropy results.}
    \label{fig:32gas234}
\end{figure}

\subsection{A damped thermoelastic rod}\label{ex:big}
More useful from a ROM perspective is when the systems under consideration represent discretizations of infinite-dimensional metriplectic dynamics.  Consider an infinite-dimensional example of the system in \cite[Section 3.1]{mielke2011} (specialized in \cite{kraaij2020}), where a 1-D elastic rod with coordinate $s\in[0,\ell]$ evolves as a damped Hamiltonian system with friction.  The dynamics of this motion are governed by the metriplectic system
\[ 
\begin{pmatrix}
\dot{q} \\ \dot{p} \\ \dot{S}
\end{pmatrix}
=
\begin{pmatrix}
\frac{p}{m} \\ V'(q) - \gamma\frac{p}{m} \\ \gamma\frac{p^2}{m^2}
\end{pmatrix}
=
\LL\nabla E(q,p,e) + \MM(q,p,e)\nabla S,
\]
where $V$ is a given potential function, $\gamma$ is a constant controlling the rate of dissipation, $\nabla S,\nabla E$ now denote $L^2$-gradients, $m=m(s)$ is the mass density, $q=q(s)$ is the position, $p=p(s)$ is the momentum, and $e=e(s)$ is the internal energy representing the conversion of mechanical energy into heat.  Explicitly, the state of the system can be described by $\xx(s) = \begin{pmatrix} q(s) & p(s) & S(s) \end{pmatrix}^\intercal$ where the functions $E$ and $S$ are given by
\[ E(p,q,e) = H(p,q) + S(e) = \int_0^\ell \lr{\frac{p(s)^2}{2m(s)} + V\lr{q(s)}} + \int_0^\ell e(s), \]
and where $H(p,q)$ is the Hamiltonian function for the system.  Notice that $H$ is no longer a conserved quantity but has been replaced by $E$ which balances energetic and entropic contributions.  A simple calculation using the definition $dF(v) = \lr{\nabla F,v}_{L^2}$ yields the gradients and operators which describe these metriplectic dynamics: it follows that $\nabla S = \begin{pmatrix}0 & 0 & 1\end{pmatrix}^\top$,
\begin{equation*}
    \LL = 
    \begin{pmatrix}
    0 & 1 & 0 \\
    -1 & 0 & 0 \\
    0 & 0 & 0
    \end{pmatrix},\quad
    \MM = \gamma
    \begin{pmatrix}
    0 & 0 & 0 \\
    0 & 1 & -\frac{p}{m} \\
    0 & -\frac{p}{m} & \left(\frac{p}{m}\right)^2
    \end{pmatrix},\quad
    \nabla E = 
    \begin{pmatrix}
    V'(q) \\
    \frac{p}{m} \\
    1
    \end{pmatrix}.
\end{equation*}

\begin{figure}
    \centering
    \begin{minipage}{0.95\textwidth}
    \includegraphics[width=\textwidth]{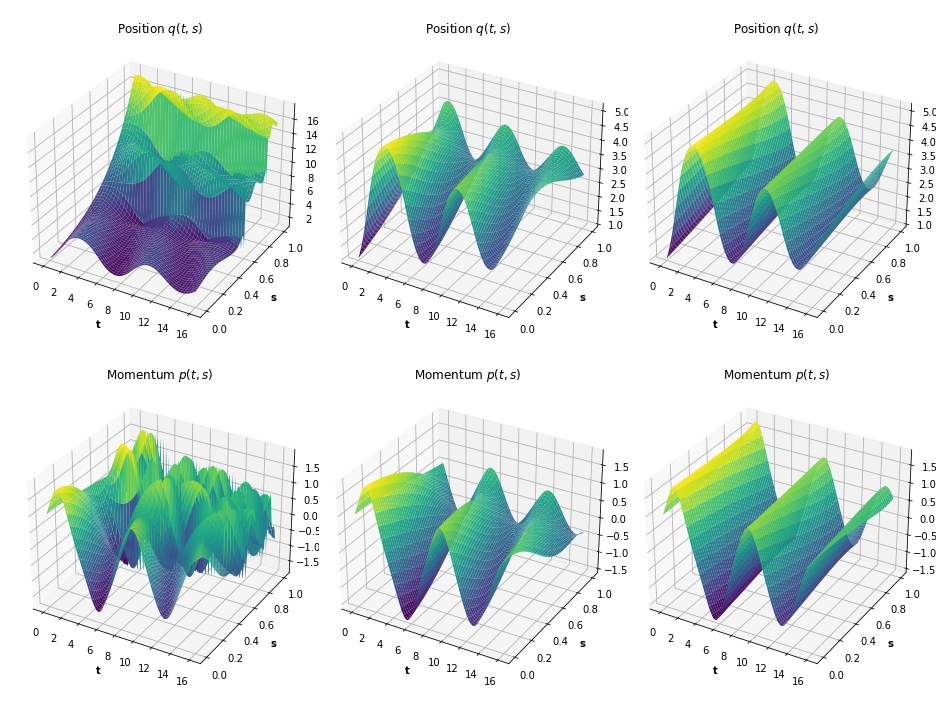}
    \end{minipage}
    \caption{The position $q(t,s)$ and momentum $p(t,s)$ functions of three qualitatively different solutions to the thermoelastic rod problem contained in the training set.}
    \label{fig:nbodiesFOM}
\end{figure}

Modeling this evolution requires an appropriate discretization of the continuous system above.  Consider a semi-discretization in the rod parameter with constant mass density $m$, so that $s$ is represented by the vector $\bb{s}\in\mathbb{R}^N$ where $N$ is the number of discretization points.  Then, the discretized state (also denoted $\xx$) becomes a $(2N+1)$-vector which evolves according to 
\[ 
\begin{pmatrix}
\dot{\qq} \\ \dot{\pp} \\ \dot{S}
\end{pmatrix}
=
\begin{pmatrix}
\frac{\pp}{m} \\ \bb{V}'(\qq) - \gamma\frac{\pp}{m} \\ \gamma\frac{\nn{\pp}^2}{m^2}
\end{pmatrix}
=
\LL\nabla E(\xx) + \MM(\xx)\nabla S,
\]
where $\nabla S = \begin{pmatrix}\bb{0} & \bb{0} & 1\end{pmatrix}^\intercal$,
\begin{equation*}
    \LL = 
    \begin{pmatrix}
    \bb{0}_{N\times N} & \bb{I} & \bb{0} \\
    -\bb{I} & \bb{0}_{N\times N} & \bb{0} \\
    \bb{0} & \bb{0} & 0
    \end{pmatrix},\quad
    \MM(\xx) = \gamma
    \begin{pmatrix}
    \bb{0}_{N\times N} & \bb{0}_{N\times N} & \bb{0} \\
    \bb{0}_{N\times N} & \bb{I} & -\frac{\pp}{m} \\
    \bb{0} & -\frac{\pp^\intercal}{m} & \left(\frac{\nn{\pp}}{m}\right)^2
    \end{pmatrix},\quad
    \nabla E(\xx) = 
    \begin{pmatrix}
    \bb{V}'(\qq) \\
    \frac{\pp}{m} \\
    1
    \end{pmatrix}.
\end{equation*}
Notice that $S^{2N+1}=E^{2N+1}=1\neq 0$ and $\MM$ decomposes as
\[ \MM = \sum_{\alpha=1}^N\bb{m}^\alpha\otimes\bb{m}^\alpha, \qquad \bb{m}^\alpha = \sqrt{\gamma}\begin{pmatrix}\bm{0} & \bb{e}_\alpha & -\frac{p_\alpha}{m} \end{pmatrix}^\intercal, \]
so that $\xih, \zh$ can be computed as
\begin{align*}
    \xih = \Lb\wedge\uu^{2N+1}, \qquad \zh\lr{\xt} = \sum_{\alpha=1}^N \Am^\alpha\lr{\xt}\otimes\Am^\alpha\lr{\xt}, \qquad \Am^\alpha\lr{\xt} = \ut\bb{m}^\alpha\lr{\xt} \wedge \uu^{2N+1}.
\end{align*}
Moreover, although $\MM$ depends on the full-order state $\xt\in\mathbb{R}^{2N+1}$, each $\bb{m}^\alpha$ is linear in the solution and so the online cost can be made independent of this number (although it will still depend on $N$, the number of terms in the sum).  More precisely, notice that 
\[\MM = \bb{C}\bb{C}^\intercal,\qquad \bb{C} = \sqrt{\gamma}\begin{pmatrix}\bm{0}_{N\times N} & \bm{I} & -\frac{\pp}{m},\end{pmatrix}^\intercal\]
so that $\bb{C}\lr{\xt}\in\mathbb{R}^{(2N+1)\times N}$ can be written as 
\[ \bb{C}\lr{\xt} = \bb{C}_0 + \bb{C}_1\lr{\xh} = \sqrt{\gamma}\begin{pmatrix}\bm{0}_{N\times N} \\ \bm{I} \\ -\frac{\pp^\intercal}{m}\end{pmatrix} + \frac{\sqrt{\gamma}}{m}\begin{pmatrix}\bm{0}_{N\times N} \\ \bm{0}_{N\times N} \\ -\lr{\uu^{N:2N}\xh}^\intercal\end{pmatrix}, \]
where $\uu^{N:2N}$ indicates the $N\times n$ matrix formed from rows $N$ to $2N$ of $\uu$.  It follows that the reduced object
\[\ut\bb{C}\lr{\xt} = \ut\bb{C}_0 + \ut\bb{C}_1\lr{\xh} = \ut\bb{C}_0 - \frac{\sqrt{\gamma}}{m}\uu^{2N+1}\otimes \uu^{N:2N}\xh,\] contains all $\ut\bb{m}^\alpha$ in its columns and requires only multiplication by $\xh$ online.  This optimization has been used throughout this example on the FOM (where $\pp = \pp_0 + (\pp-\pp_0)$) as well as all ROMs.  In addition, all simulations use the potential $V(q) = \cos q$ along with constants $\gamma=8$, $\ell=1$ and $N=250$.

The performance of each ROM in this case is evaluated using a family of trajectories with parameterized initial conditions.  More precisely, it is assumed that $\bb{x}_0$ satisfies the initial position and momentum conditions
\begin{align*}
    \bb{q}_0(\bb{s}) &= e^{\mu_1\bb{s}}, \qquad \mu_1 \in [-0.2,5.2], \\
    \bb{p}_0(\bb{s}) &= \frac{1}{1+\mu_2\bb{s}^2}, \qquad \mu_2\in[-1,1],
\end{align*}
along with some initial entropy $S_0 \in [1,3]$. Some representative solutions to the thermoelastic rod system with this type of initial data can be found in Figure~\ref{fig:nbodiesFOM}.  As in the previous example, 25 uniformly random instances of $\begin{pmatrix}\mu_1&\mu_2&S_0\end{pmatrix}^\intercal$ are drawn and used to form the initial conditions used for training the POD map $\uu$.  Note that the snapshot matrix $\bb{Y}$ includes snapshots of the shifted solution $\xx-\xx_0$ as well as the gradients $\nabla E(\xx)$ collected from the interval $[0,8]$ in $t$-increments of $0.02$. Again, the constant vector $\nabla S$ is included as the last column of $\bb{Y}$.  It is interesting to note that the first eigenvalue of the snapshot matrix is much larger than the rest, but the remaining eigenvalues decay relatively slowly until roughly $n=350$ (see Figure~\ref{fig:evals}).  This indicates that these dynamics cannot be captured well with only a few linear POD basis functions.

\begin{figure}
    \centering
    \includegraphics[width=\textwidth]{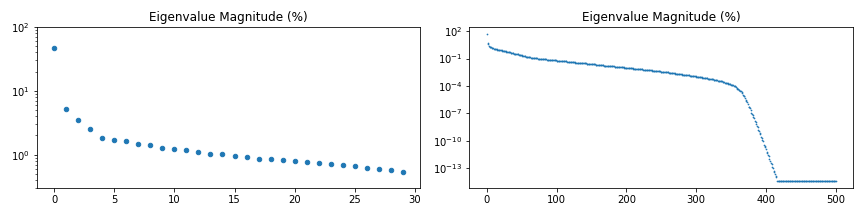}
    \caption{Eigenvalue plots corresponding to the snapshot matrix in the thermoelastic rod example.  The $y$-axis displays the magnitude of each eigenvalue as a percentage of the total sum.}
    \label{fig:evals}
\end{figure}

After training $\uu$, the ROMs are evaluated online using the initial data corresponding to the parameter $\begin{pmatrix}0.65&-0.1&1.9\end{pmatrix}^\intercal$ not included in the training set.  The results of this procedure are tabulated in Table~\ref{tab:nbody} and illustrated in Figures~\ref{fig:nbodies16} and \ref{fig:nbodies48}.  Clearly, the greatest advantage of the metriplectic SP-ROM is its ability to preserve the energy conservation and entropy growth of the original system independently of $n$, leading to much greater accuracy and stability over time.  Conversely, there appears to be little advantage to employing the more expensive SP-ROM over the cheaper, matrix-oriented EH-ROM when the integration takes place over small times, as there is not enough accumulation from the violation of the compatibility conditions \eqref{eq:compatibility} to significantly harm model performance.  Note that the simple and straightforward G-ROM is unstable and inferior in almost every case.  Since it has no knowledge of the internal structure of the system, it cannot accurately infer the original metriplectic dynamics.

\begin{table}
\begin{center}\label{tab:nbody}
\resizebox{\columnwidth}{!}{%
\def\arraystretch{1.25}%
\begin{tabular}{|c|c|ccccc|c|c|ccccc|}
\hline
$T$ & $n$ & Method & $\mathcal{E}_r$ \% & $\mathcal{E}_\infty$ & $\nn{E(T)-E_0}$ & Time (s) & $T$ & $n$ & Method & $\mathcal{E}_r$ \% & $\mathcal{E}_\infty$ & $\nn{E(T)-E_0}$ & Time (s) \\ \hline
\multirow{10}{*}{8} & - & FOM & - & - & $1.918\times 10^{-11}$ & 0.1614 & \multirow{10}{*}{16} & - & FOM & - & - & $1.651\times 10^{-11}$ & 0.2766 \\ \cline{2-7} \cline{9-14}
& \multirow{3}{*}{10} & SP-ROM & 9.015 & 22.28 & $2.814\times 10^{-12}$ & 0.1965 & & \multirow{3}{*}{10} & SP-ROM & 8.166 & 26.49 & $1.648\times 10^{-12}$ & 0.4123 \\
& & EH-ROM & 4.896 & 7.980 & 6.687 & 0.06249 & & & EH-ROM & 4.134 & 7.980 & 36.43 & 0.1039 \\
& & G-ROM & 8.336 & 21.65 & 16.99 & 0.1091 & & & G-ROM & 8.394 & 21.65 & 59.49 & 0.1975 \\ \cline{2-7} \cline{9-14}
& \multirow{3}{*}{20} & SP-ROM & 4.020 & 7.938 & $1.108\times 10^{-12}$ & 0.1762 & & \multirow{3}{*}{20} & SP-ROM & 3.565 & 10.71 & $5.571\times 10^{-12}$ & 0.4895 \\
& & EH-ROM & 3.625 & 5.708 & 5.686 & 0.05898 & & & EH-ROM & 2.881 & 5.708 & 21.06 & 0.1107 \\
& & G-ROM & 4.864 & 15.49 & 9.108 & 0.1186 & & & G-ROM & 5.304 & 15.49 & 61.26 & 0.1948 \\ \cline{2-7} \cline{9-14}
& \multirow{3}{*}{40} & SP-ROM & 0.8734 & 0.5371 & $4.121\times 10^{-12}$ & 0.2523 & & \multirow{3}{*}{40} & SP-ROM & 0.9430 & $1.509$ & $4.234\times 10^{-12}$ & 0.4403 \\
& & EH-ROM & 0.8767 & 0.5469 & 0.8563 & 0.07977 & & &  EH-ROM & 0.9339 & 1.694 & 4.457 & 0.1378 \\
& & G-ROM & 1.001 & 0.4779 & 1.555 & 0.1339 & & & G-ROM & 1.207 & 1.029 & 9.142 & 0.2258 \\ \hline
\multirow{10}{*}{48} & - & FOM & - & - & $1.253\times 10^{-11}$ & 0.6041 & \multirow{10}{*}{96} & - & FOM & - & - & $3.439\times 10^{-12}$ & 1.003 \\ \cline{2-7} \cline{9-14}
& \multirow{3}{*}{10} & SP-ROM & 9.197 & 42.97 & $1.708\times 10^{-11}$ & 0.6814 & & \multirow{3}{*}{10} & SP-ROM & 10.11 & 44.97 & $8.413\times 10^{-12}$ & 0.9792 \\
& & EH-ROM & 18.20 & 154.0 & 280.6 & 0.2015 & & & EH-ROM & 92.49 & 846.9 & 1082 & 0.4351 \\
& & G-ROM & 15.20 & 122.6 & 297.3 & 0.3931 & & & G-ROM & 72.28 & 787.7 & 1190 & 0.7518 \\ \cline{2-7} \cline{9-14}
& \multirow{3}{*}{20} & SP-ROM & 4.456 & 21.09 & $2.643\times 10^{-12}$ & 0.6928 & & \multirow{3}{*}{20} & SP-ROM & 5.013 & 6.056 & $8.811\times 10^{-13}$ & 0.8788 \\
& & EH-ROM & 9.079 & 68.68 & 103.6 & 0.2628 & & & EH-ROM & 29.25 & 252.3 & 320.5 & 0.4126 \\
& & G-ROM & 107.3 & 1498 & 3051 & 0.5186 & & & G-ROM & - & - & - & -\\ \cline{2-7} \cline{9-14}
& \multirow{3}{*}{40} & SP-ROM & 1.302 & 5.538 & $4.832\times 10^{-13}$ & 0.6735 & & \multirow{3}{*}{40} & SP-ROM & 1.495 & 6.056 & $2.842\times 10^{-12}$ & 1.427 \\
& & EH-ROM & 2.063 & 14.42 & 21.07 & 0.2682 & & & EH-ROM & 5.369 & 41.49 & 49.38 & 0.5114 \\
& & G-ROM & 3.908 & 32.51 & 67.35 & 0.5442 & & & G-ROM & 18.69 & 168.8 & 240.3 & 0.9306 \\ \hline
\end{tabular}%
}
\end{center}
\medskip
\caption{Results of the thermoelastic rod experiment.  Dashes indicate ``not applicable'' when reporting the FOM, and lack of convergence when reporting the ROMs.}
\end{table}

\begin{figure}
    \centering
    \includegraphics[width=\textwidth]{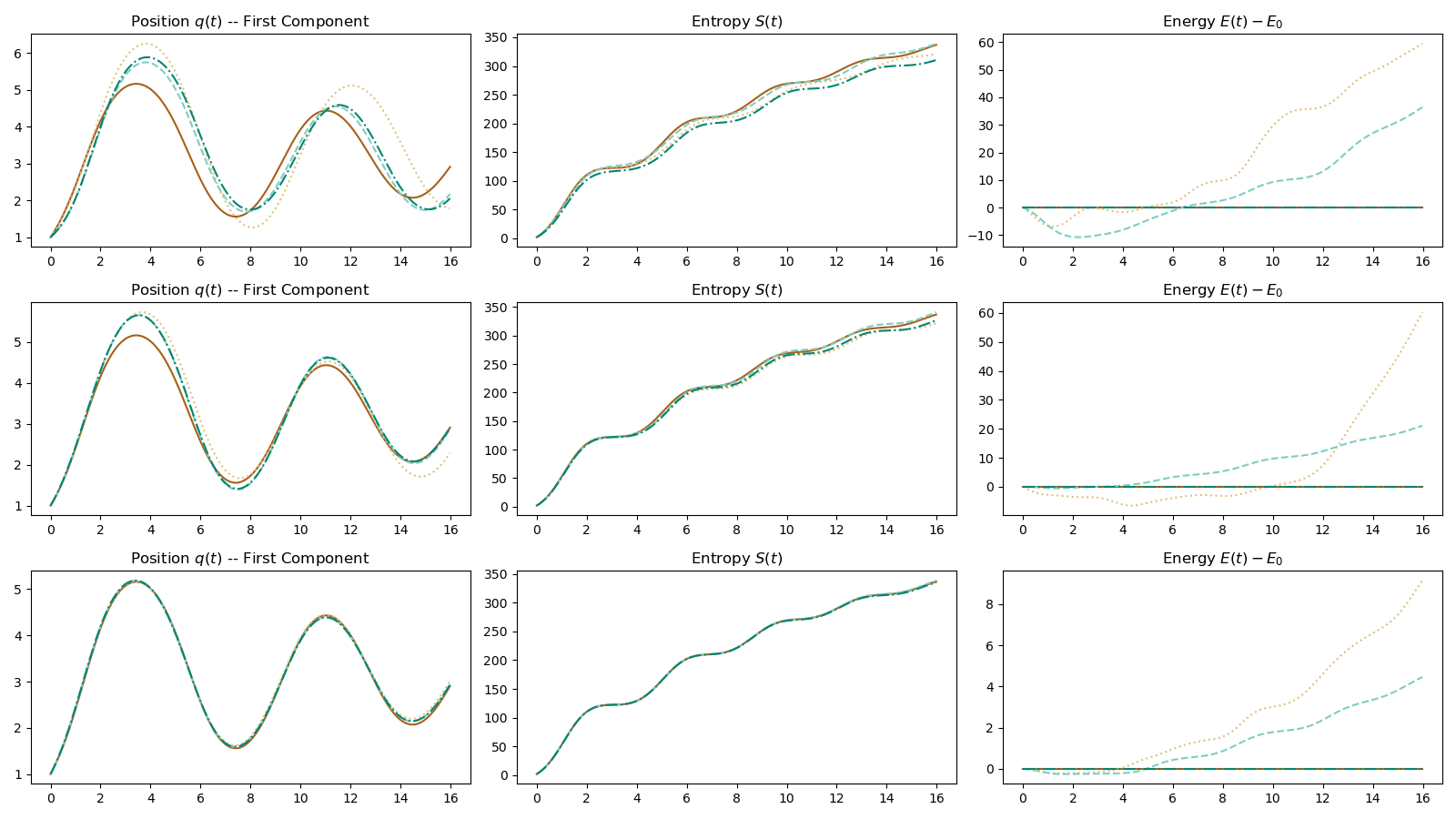}
    \caption{A comparison of ROM solutions for the 501-dimensional thermoelastic rod example when $T=16$ and $n=10,20,40$, respectively.  Plotted are the \textcolor{exact}{Exact Solution (solid)}, \textcolor{no}{G-ROM (dotted)}, \textcolor{old}{EH-ROM (dashed)}, and \textcolor{new}{SP-ROM (dot-dashed)}.  Observe the convergence as established in Theorem~\ref{thm:conv}.}
    \label{fig:nbodies16}
\end{figure}

\begin{figure}
    \centering
    \includegraphics[width=\textwidth]{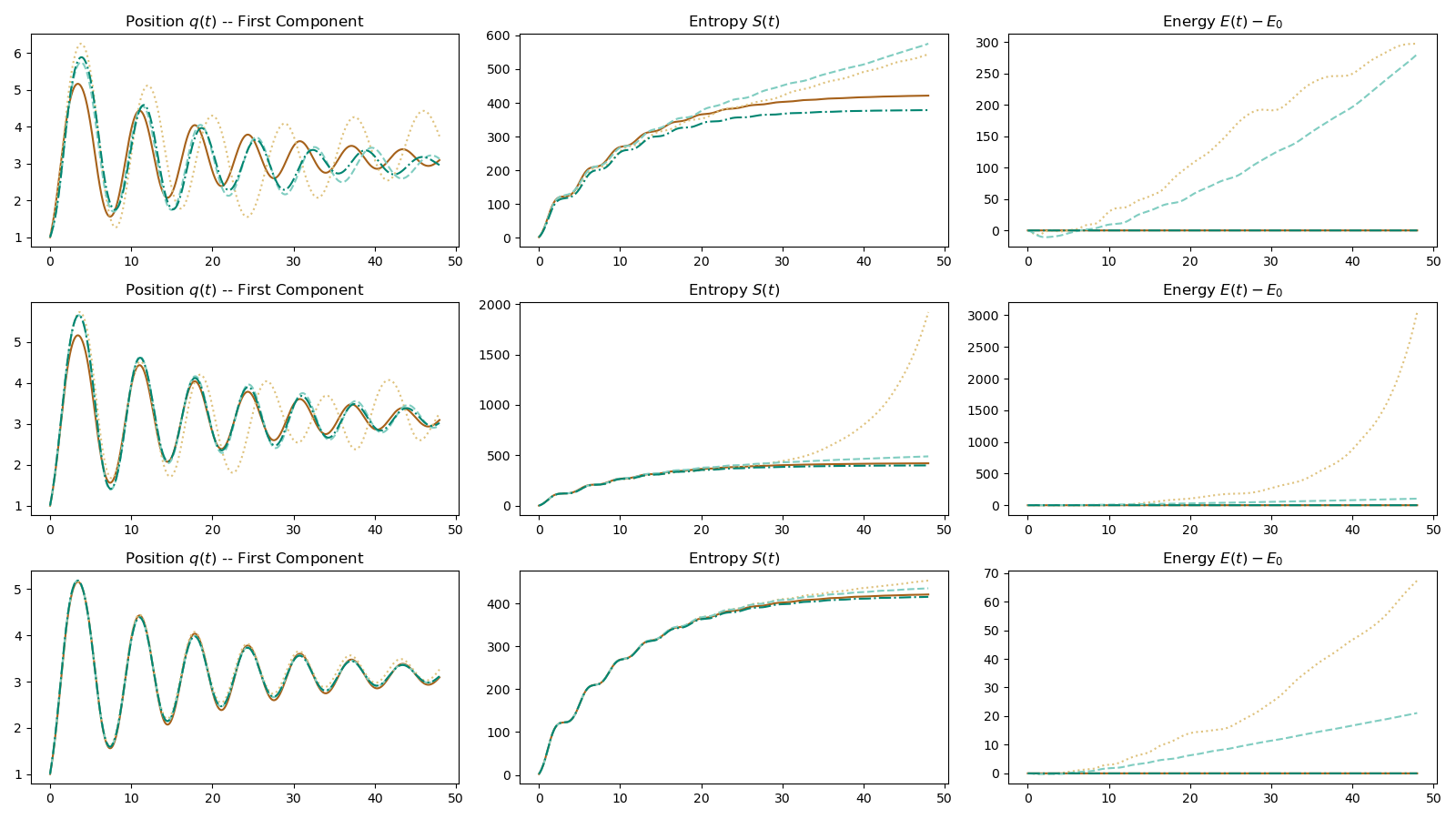}
    \caption{A comparison of ROM solutions for the 501-dimensional thermoelastic rod example when $T=48$ and $n=10,20,40$, respectively.  Plotted are the \textcolor{exact}{Exact Solution (solid)}, \textcolor{no}{G-ROM (dotted)}, \textcolor{old}{EH-ROM (dashed)}, and \textcolor{new}{SP-ROM (dot-dashed)}.  Note that only the \textcolor{new}{SP-ROM} produces reasonable energy and entropy results.}
    \label{fig:nbodies48}
\end{figure}

It is also useful to investigate the long-term stability of these ROMs.  As the EH-ROM and the G-ROM are not truly metriplectic, it is expected that their performance will decay as the interval of integration moves far away from the training data.  This is tested using the same experiment as above by varying the right endpoint of the temporal integration (recall that $\uu$ is trained only on snapshots coming from the interval $[0,8]$). Table~\ref{tab:stability} shows the results of integrating over ranges $[0,T]$ where $T=2^k, 3\leq k\leq 9$.  As expected, the metriplectic SP-ROM is quite stable, while the others eventually break down.  It is interesting that the naive G-ROM is unpredictable, exhibiting better stability when $n=10$ than when $n=20$.  


\begin{table}[]\small
\begin{center}
\begin{tabular}{|ccc|ccccccl|}
\hline
\multicolumn{3}{|c|}{\multirow{2}{*}{$\mathcal{E}_r$ \%}} & \multicolumn{7}{c|}{$T$} \\ \cline{4-10} 
\multicolumn{3}{|c|}{} & \multicolumn{1}{c|}{8} & \multicolumn{1}{c|}{16} & \multicolumn{1}{c|}{32} & \multicolumn{1}{c|}{64} & \multicolumn{1}{c|}{128} & \multicolumn{1}{c|}{256} & 512 \\ \hline
\multicolumn{1}{|c|}{\multirow{9}{*}{$n$}} & \multicolumn{1}{c|}{\multirow{3}{*}{5}} & SP-ROM & 19.28 & 17.96 & 17.18 & 17.44 & 17.55 & 17.36 & 17.18 \\
\multicolumn{1}{|c|}{} & \multicolumn{1}{c|}{} & EH-ROM & 11.36 & 10.02 & 12.34 & 668.0 & - & - & - \\
\multicolumn{1}{|c|}{} & \multicolumn{1}{c|}{} & G-ROM & 10.21 & 11.69 & 12.88 & 70.01 & 322.4 & - & - \\ \cline{2-10} 
\multicolumn{1}{|c|}{} & \multicolumn{1}{c|}{\multirow{3}{*}{10}} & SP-ROM & 9.015 & 8.166 & 8.425 & 9.672 & 10.29 & 10.51 & 10.59 \\
\multicolumn{1}{|c|}{} & \multicolumn{1}{c|}{} & EH-ROM & 4.896 & 4.134 & 7.815 & 34.64 & 160.6 & - & - \\
\multicolumn{1}{|c|}{} & \multicolumn{1}{c|}{} & G-ROM & 8.336 & 8.394 & 8.409 & 27.23 & 509.0 & - & - \\ \cline{2-10} 
\multicolumn{1}{|c|}{} & \multicolumn{1}{c|}{\multirow{3}{*}{20}} & SP-ROM & 4.020 & 3.565 & 3.952 & 4.750 & 5.118 & 5.231 & 5.262 \\
\multicolumn{1}{|c|}{} & \multicolumn{1}{c|}{} & EH-ROM & 3.624 & 2.881 & 4.858 & 14.54 & 51.36 & 922.8 & - \\
\multicolumn{1}{|c|}{} & \multicolumn{1}{c|}{} & G-ROM & 4.864 & 5.304 & 10.77 & - & - & - & - \\ \hline
\end{tabular}\label{tab:stability}
\end{center}
\medskip
\caption{Long-time integration results for the thermoelastic rod example.  Values are the relative error in percentage form, and dashes indicate when the solver does not converge.}
\end{table}




\section{Conclusion}

A new strategy for the model reduction of metriplectic systems has been proposed and shown to  guarantee a strong form of the first and second laws of thermodynamics.  By preserving the original metriplectic structure at the algebraic level, the proposed ROM is able to produce more realistic energy and entropy profiles than other POD-ROMs which cannot guarantee structure-preservation.  It has been shown that the metriplectic POD-ROM conserves energy to arbitrary precision regardless of the reduced dimension and converges to the true solution as this dimension increases.  As the proposed ROM is trained similarly to standard POD-ROMs, it could be useful as a drop-in replacement for metriplectic problems demanding physically realistic solutions which must remain stable and energy-conserving.   Future work will investigate applications to more complex problems such as those mentioned in Section~\ref{sec:rw}, as well as effective methods such as Discrete Empirical Interpolation for making the metriplectic ROM completely independent of the full-order dimension.  As several important problems in fluid mechanics are also known to have a metriplectic form (see e.g. \cite{ottinger2006}), it is especially interesting to consider the application of these techniques to realistic oceanic and atmospheric models which require strict adherence to physical laws.





\section*{Acknowledgments}
This work is partially supported by U.S. Department of Energy Scientific Discovery through Advanced Computing under grants DE-SC0020270 and DE-SC0020418.

\bibliographystyle{ieeetr}
\bibliography{biblio}

\end{document}